\documentclass[a4paper,oneside,11pt]{article}

\usepackage{amsmath}
\usepackage{amssymb}
\usepackage[utf8]{inputenc} 
\usepackage{accents}
\usepackage{mathtools}
\usepackage{amsthm}
\usepackage[english]{babel}
\usepackage{letltxmacro}
\usepackage{xparse}
\usepackage[dvipsnames,svgnames,x11names]{xcolor}
\usepackage{graphicx}
\usepackage[inline]{enumitem}
\usepackage{yhmath}

\usepackage{fullpage}
\usepackage{fancyhdr}
\usepackage{ebgaramond}
\usepackage{euscript}
\usepackage{mathrsfs}
\usepackage{accents}
\usepackage[bbgreekl]{mathbbol}
\usepackage[euler]{textgreek}
\usepackage[scr=boondoxo]{mathalfa}
\DeclareSymbolFontAlphabet{\mathbbm}{bbold}
\DeclareSymbolFontAlphabet{\mathbb}{AMSb}%
\usepackage{mleftright}
\usepackage{dirtytalk}
\usepackage{shuffle}
\usepackage{scalerel}
\usepackage{tensor}
\usepackage{soul}
\setul{0pt}{0.4pt}

\mathchardef\mhyphen="2D

\usepackage{tikz}
\usetikzlibrary{decorations.pathreplacing}
\usetikzlibrary{decorations.markings}
\usetikzlibrary{decorations.fractals}
\usetikzlibrary{lindenmayersystems}
\usepackage{pgfplots}
\usepackage{wrapfig}
\usepackage{caption}
\usepackage[percent]{overpic}
\usepackage{tikz-cd}

\usepackage{aliascnt}

\usepackage[
  colorlinks=true,
  linkcolor=BrickRed,
  citecolor=ForestGreen,
  urlcolor=ProcessBlue
]{hyperref}

\addto\extrasenglish{
	
}
\addto\extrasenglish{
	
}

\newcommand{\mynewtheorem}[2]{
\newaliascnt{#1}{dummy}
\newtheorem{#1}[#1]{#2}
\aliascntresetthe{#1}
\expandafter\def\csname #1autorefname\endcsname{#2}
}

\newtheoremstyle{note}
{\topsep}   
{\topsep}   
{}  
{0pt}       
{\itshape\bfseries} 
{.}         
{5pt plus 1pt minus 1pt} 
{{\color{imperialTangerine}\thmname{#1}\emph{\thmnumber{ #2}}}\hspace{0.1em}\textnormal{\thmnote{ (#3)}}}     

\theoremstyle{plain}
\mynewtheorem{thm}{Theorem}
\mynewtheorem{thmdef}{Theorem/Definition}
\mynewtheorem{prop}{Proposition}
\mynewtheorem{cor}{Corollary}
\mynewtheorem{lem}{Lemma}
\mynewtheorem{conj}{Conjecture}
\theoremstyle{definition}
\mynewtheorem{defn}{Definition}
\mynewtheorem{expl}{Example}
\mynewtheorem{prob}{Problem}
\mynewtheorem{ques}{Question}
\mynewtheorem{cond}{Condition}
\mynewtheorem{conv}{Convention}
\mynewtheorem{defthm}{Definition/Theorem}
\theoremstyle{remark}
\mynewtheorem{rem}{Remark}

\usepackage{iftex}
\iftutex
\usepackage{unicode-math}
\setmathfontface\altgrfont{GFS Artemisia Italic}[Scale=MatchLowercase]

\else
\usepackage[T1]{fontenc}
\DeclareSymbolFont{altgr}{OML}{antt}{m}{it}
\DeclareMathSymbol{\sko }{\mathord}{altgr}{"0E}
\fi

\makeatletter
\def\upintkern@{\mkern-7mu\mathchoice{\mkern-3.5mu}{}{}{}}
\def\upintdots@{\mathchoice{\mkern-4mu\@cdots\mkern-4mu}%
{{\cdotp}\mkern1.5mu{\cdotp}\mkern1.5mu{\cdotp}}%
{{\cdotp}\mkern1mu{\cdotp}\mkern1mu{\cdotp}}%
{{\cdotp}\mkern1mu{\cdotp}\mkern1mu{\cdotp}}}

\newcommand{\UpMultiIntegral}[1]{%
\edef\ints@c{\noexpand\upintop
	\ifnum#1=\z@\noexpand\upintdots@\else\noexpand\upintkern@\fi
	\ifnum#1>\tw@\noexpand\upintop\noexpand\upintkern@\fi
	\ifnum#1>\thr@@\noexpand\upintop\noexpand\upintkern@\fi
	\noexpand\upintop
	\noexpand\ilimits@
}%
\futurelet\@let@token\ints@a
}
\makeatother

\DeclareFontFamily{OMX}{mdbch}{}
\DeclareFontShape{OMX}{mdbch}{m}{n}{ <->s * [0.8]  mdbchr7v }{}
\DeclareFontShape{OMX}{mdbch}{b}{n}{ <->s * [0.8]  mdbchb7v }{}
\DeclareFontShape{OMX}{mdbch}{bx}{n}{<->ssub * mdbch/b/n}{}

\DeclareSymbolFont{uplargesymbols}{OMX}{mdbch}{m}{n}
\SetSymbolFont{uplargesymbols}{bold}{OMX}{mdbch}{b}{n}
\DeclareMathSymbol{\upintop}{\mathop}{uplargesymbols}{82}
\DeclareMathSymbol{\upointop}{\mathop}{uplargesymbols}{"48}

\DeclareFontEncoding{MDB}{}{}
\DeclareFontFamily{MDB}{mdbch}{}
\DeclareFontShape{MDB}{mdbch}{m}{n}{ <->s * [0.8]  mdbchrmb }{}
\DeclareFontShape{MDB}{mdbch}{b}{n}{ <->s * [0.8]  mdbchbmb }{}
\DeclareFontShape{MDB}{mdbch}{bx}{n}{<->ssub * mdbch/b/n}{}
\DeclareFontSubstitution{MDB}{cmr}{m}{n}
\DeclareSymbolFont{mathdesignB}{MDB}{mdbch}{m}{n}%
\SetSymbolFont{mathdesignB}{bold}{MDB}{mdbch}{b}{n}%
\DeclareMathSymbol{\upintclockwise}{\mathop}{mathdesignB}{128}
\DeclareMathSymbol{\upointclockwise}{\mathop}{mathdesignB}{130}
\DeclareMathSymbol{\upointctrclockwise}{\mathop}{mathdesignB}{132}
\DeclareMathSymbol{\upoiint}{\mathop}{mathdesignB}{134}
\DeclareMathSymbol{\upoiiint}{\mathop}{mathdesignB}{136}

\makeatletter
\newcommand{\upint}{\DOTSI\upintop\ilimits@}
\newcommand{\upoint}{\DOTSI\upointop\ilimits@}
\makeatother

\renewcommand{\int}{\upint}
\newcommand{\dif}{\,\mathrm{d}}
\newcommand{\Lim}{\operatorname{Lim}}
\newcommand{\Z}{\mathbb{Z}}
\newcommand{\X}{\mathbb{X}}

\newcommand{\scal}[1]{\langle #1 \rangle}
\newcommand{\R}{\mathbb{R}}
\newcommand{\Q}{\mathbb{Q}}

\newcommand{\eps}{\varepsilon}

\newcommand{\cl}{\mathrm{cl}}

\newcommand{\Hol}[1]{#1\text{-}\mathrm{H\ddot ol}}

\usepackage{authblk}

\title{\textsc{Locality of rough path lifts}}

\author[1]{Ilya Chevyrev\thanks{\texttt{ichevyrev@gmail.com}}}
\author[1]{Emilio Ferrucci\thanks{\texttt{emilio.ferrucci@sissa.it}}}

\affil[1]{SISSA, Trieste, Italy}
\date{\today}

\usetikzlibrary{decorations.markings}
\usepackage[toc,page]{appendix}

\def\bb1{\mathbbm{1}}

\def\A{\mathbb{A}}
\def\cH{\mathcal{H}}
\def\cB{\mathcal{B}}

\def\cI{\mathcal{I}}

\def\cC{\mathcal{C}}
\def\cE{\mathcal{E}}
\def\cF{\mathcal{F}}

\def\P{\mathbb{P}}
\def\E{\mathbb{E}}

\newcommand{\CHol}[1]{C^{#1\text{-}\mathrm{H\ddot ol}}}
\def\bbR{\mathbb{R}}
\def\bX{\boldsymbol{X}}
\newcommand{\be}{\mathbf{e}}

\pgfplotsset{compat=1.18}

\begin{document}

\maketitle 

\vspace{-3ex}
\begin{abstract}
Every H\"older continuous path $X$ admits a geometric rough path lift $\bX$ by the Lyons--Victoir extension theorem. A natural question that emerges when lifting more than one path segment at once is that of \emph{locality}, namely whether the lift $\bX_{s,t}$ only depends on the increments $X_{s,u}$, $u \in [s,t]$. We investigate the locality of rough path lifts in deterministic and stochastic settings.

On the deterministic side, we show that no local, homogeneous rough path lift can be defined on $\gamma$-H\"older paths for all $\gamma\leq 1/2$.
More strongly, we show that no L\'evy area can be defined which is at the same time bounded, with no further regularity assumptions, and either local and homogeneous or time translation--invariant. We moreover show that the boundedness requirement is sharp: an unbounded, local, time translation--invariant, and bilinear L\'evy area can be defined on all continuous paths.

On the stochastic side, we classify all local, square-integrable rough path lifts of $d$-dimensional fractional Brownian motion with Hurst parameter $H \in (0,1/2]$. For $H \leq 1/4$, we show that no such lifts exist when $d \geq 2$,
while for $H>1/4$, we show that all such lifts are stochastic translations of the canonical rough path. We further refine the classification by requiring invariance in law under time translation, scaling, and coordinate permutation, and show that only the canonical lift satisfies these constraints except at \(H=1/3\), for which there is a one-parameter family of lifts.
Finally, we present an argument showing that scale-invariance forces a local $\eta$-H\"older rough path lift of fractional Brownian motion to be square-integrable, allowing scale-invariance to replace square-integrability in our classification and non-existence results.
\end{abstract}

\tableofcontents

\section{Introduction}\label{sec:intro}

A rough path over an $\alpha$-H\"older path $X \colon [0,T] \to \R^d$
is a collection of two-parameter functions $\bX^{(n)}_{s,t}$, $0 \leq s \leq t \leq T$, which plays the role of iterated integrals
\begin{equation}\label{eq:iterated_integral}
\int_{s < u_1 < \ldots < u_n < t} \dif X_{u_1} \otimes \cdots \otimes \dif X_{u_n}\;.
\end{equation}
Introduced by Lyons \cite{Lyo98}, rough paths are a powerful tool to solve differential equations driven by $X$ when $\alpha \leq \frac 12$.

When $\alpha>\frac12$, there is a canonical and unique choice for $\bX^{(n)}_{s,t}$ given by Young integration \cite{You36, Lyo94}.
For $\alpha \leq \frac 12$, rough path lifts are no longer unique, but a fundamental result of Lyons--Victoir \cite{LV07} is that a rough path exists above every path.
This result is further refined in \cite{TaZa20, BZ22_sewing} by making the lift constructive and even continuous
(see \cite{Hai14, CZ20} for related results in regularity structures).
All of these constructions, however, are non-canonical and rely on arbitrary choices.

Since $\smash{\bX^{(n)}_{s,t}}$ is an abstraction of the (otherwise undefined) iterated integral \eqref{eq:iterated_integral},
a natural requirement of any rough path lift is that it be \emph{local}, meaning that $\smash{\bX^{(n)}_{s,t}}$ depends only on the increments of $X$ in the interval $[s,t]$.
In the same spirit, one can ask that the lift is \emph{time translation--invariant}, meaning that $\smash{\bX^{(n)}_{s,t} = \bX^{(n)}_{u,v}}$ whenever the increments of $X$ on $[s,t]$ and $[u,v]$ are equal.
See \autoref{fig:locality} for an illustration of these concepts.
These properties are satisfied when the integral \emph{is} defined canonically, e.g.\ in the Young regime.
The abstract lifts provided in \cite{LV07, TaZa20, BZ22_sewing} do \emph{not} satisfy either property since data from $X$ on the whole interval is needed to define $\bX$ on any subinterval.
A natural question is therefore whether rough path lifts exist that are local and/or time translation--invariant.
In this article, we investigate this question in a deterministic and stochastic setting.

\begin{figure}[h]
	\centering
	\includegraphics[width=0.5\textwidth]{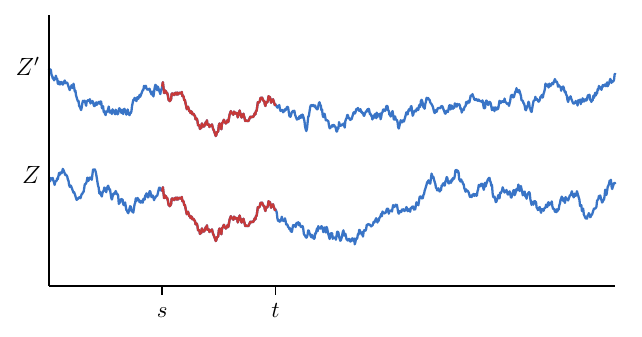}%
	\includegraphics[width=0.5\textwidth]{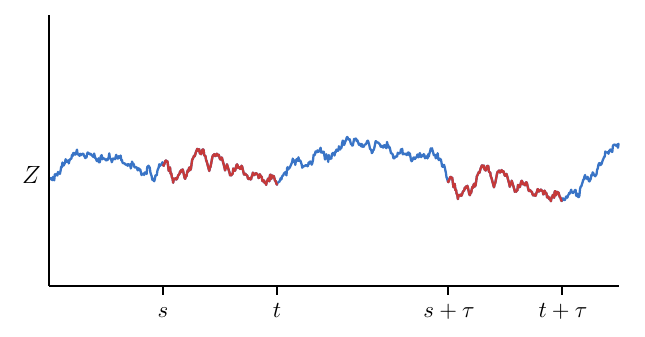}
	\vspace{-4ex}
	\caption{In the plot on the left, a local choice of area $\A$ should satisfy $\A_{s,t}(Z) = \A_{s,t}(Z')$ since $Z_u - Z_s = Z_u' - Z_s'$ for $u \in [s,t]$.
	In the plot on the right, a time translation--invariant area should satisfy $\A_{s,t}(Z) = \A_{s+\tau,t+\tau}(Z)$ since $Z_u - Z_s = Z_{u+\tau} - Z_{s + \tau}$ for $u \in [s,t]$.
	(The paths should be thought of as being multidimensional.)}
	\label{fig:locality}
\end{figure}
\vspace{-3ex}

\paragraph{Deterministic lifts.}
Our first contribution is to show that, under further mild assumptions (e.g.\ bilinearity), no local lift is possible for $\alpha\leq 1/2$.
Specifically, we show that there cannot be a rough path lift of planar $\frac 12$-H\"older paths that is either (a) local and homogeneous (the latter just means that the map $(X, Y) \mapsto \int \! X \dif Y$ behaves $\mathbb R$-bilinearly) or (b) time translation--invariant.
In fact, we replace the notion of rough path lift with the weaker notion of an \say{area process}, which we require to be merely bounded but not H\"older-continuous.
See \autoref{thm:main_determ} for a precise statement.
Our proof is obtained by considering variants of the L{\'e}vy C curve \cite{LevyC}, \cite[Ch.\ 12]{Edg04},
first introduced by Cesaro \cite{cesaro} and Faber \cite{faber}; see \autoref{fig:levyC}.

A central tool in the construction of rough path lifts and rough integrals is the concept of sewing maps (see \autoref{subsec:sewing}).
Sewing maps were introduced in \cite{Gub04,Feyel06_sewing} for H\"older exponents $\gamma>1$,
in which case there is only one sewing map,
and in \cite{BZ22_sewing} for $\gamma\leq 1$.
As a consequence of our non-existence results,
we prove in \autoref{cor:local_sewing} and \autoref{cor:time_sewing} that there are no homogeneous and local or time translation--invariant sewing maps for $\gamma\leq 1$, answering the analogous question of locality in \cite[Remark 2.10]{BZ22_sewing} in the negative.

We conclude \autoref{sec:determ} by showing that, for non-existence, the boundedness requirement of area processes is sharp.
Specifically, we show that there is a version of sewing maps on continuous germs that respect the usual algebraic (but not analytic) constraints  and which is linear, local, and time translation--invariant (but necessarily unbounded);
see \autoref{thm:algebraic_local_sewing} and \autoref{cor:algebraic_area}.
We also show in \autoref{thm:algebraic_local_sewing} that, for discontinuous germs, such local sewing maps still exist but time translation--invariant ones do not.
While not directly related to rough path theory due to lack of regularity requirements, we find these results somewhat surprising, especially that time translation-invariance is the sole obstruction to unbounded sewing maps on arbitrary germs.

\paragraph{Stochastic lifts.}

When the path $X$ is random, the probabilistic structure often leads to a (more) canonical choice of rough path lift.
Stochastic processes for which rough paths have been constructed and studied include Gaussian processes \cite{CQ02,FV10b,CF10,FGGR16}, (c\`adl\`ag) semimartingales \cite{FrizVictoir08BDG,CF19}, and Markov processes \cite{BHL02,FV08, Cass_Ogrodnik_17,CO18};
see the books \cite{LQ02, FV10, FH20} for further details.
In all of these examples, a \say{canonical} lift above $X$ is given as the limit in probability of the lift of bounded variation (e.g.\ piecewise linear) approximations of $X$.
Other lifts also naturally arise, e.g.\ the It\^o lift for semimartingales, or
lifts that are neither Stratonovich nor It\^o that emerge in various physical systems \cite{FGL15,KM16,Gottwald_Melbourne_24,Engel_Friz_Orenshtein_26}.
It is easy to verify that the above-mentioned stochastic lifts are local, understood in terms of measurability with respect to local increments of $X$.
(This does not contradict the non-existence of local rough path lifts for $\alpha\leq 1/2$ in the deterministic setting of \autoref{sec:determ} because, in the stochastic setting, one is only asking for a lift to exist almost surely.)

The stochastic process we focus on in this article is fractional Brownian motion with Hurst parameter $H\in (0,1)$ ($H$-fBm),
which is $\alpha$-H\"older continuous for $\alpha<H$.
It is useful to distinguish three regimes: $H>1/2$, $H\in (1/4,1/2]$, and $H\leq 1/4$. The case $H>1/2$ is the Young regime in which case rough path lifts are unique.
For $H\in (\frac14,\frac12]$, it was first shown in \cite{CQ02} that piecewise linear approximations of $H$-fBm converge in $L^2$ to a rough path lift; this is called the canonical lift.
For $H\leq 1/4$, the same work shows that these approximations have no subsequential limit in probability. Other constructions do give rough path lifts for $H\leq 1/4$ by exploiting the Gaussian Volterra structure \cite{Unt10, NuaTin11}, but they are not local in the above sense and, for $H>1/4$, do not reduce to the canonical lift. It was shown in \cite{Hai25} that another natural way to treat the regime $H \leq 1/4$ is by ‘variance renormalisation’, which leads to convergence in distribution to a
random rough path above the constant zero-path.

Our main contribution in the stochastic setting is \autoref{thm:localfbm}, which classifies all local, square-integrable lifts of $d$-dimensional $H$-fBm for $H\in (0,1/2]$.
We show that, for $H\leq 1/4$ and $d \geq 2$, no such lifts exist.
On the other hand, for $H > 1/4$ all such lifts $\bX$ are suitable translations of the canonical lift $\overline\bX$, given by Wiener integrals against a deterministic rough path above the zero-path.
We prove this by expanding the difference $\bX-\overline\bX$ in Wiener chaos and using the Chen identity to show that all its chaos projections vanish except for the zero-th (at degree $2$) and zero-th and first (at degree $3$).
In \autoref{cor:inv} we further classify which of these lifts are stationary, scale-invariant, or invariant with respect to permutation of the spatial coordinates. An unexpected result is that, apart from the canonical lift, the only other lifts for which all three properties hold occur at $H = \frac 13$, and there is a one-dimensional continuum of them.

We conjecture that the $L^2$ assumption in \autoref{thm:localfbm} is unnecessary, and that our classification of local rough path lifts of $H$-fBm holds without it.
We end by presenting a proof, suggested by GPT 5.6 Sol Pro, that if the rough path lift is further assumed to be scale-invariant, then it is automatically in $L^2$, see \autoref{thm:scale_integrability}.
This in particular implies that scale-invariance can replace square-integrability in the classification \autoref{thm:localfbm}.

\paragraph{Acknowledgements}
Both authors gratefully acknowledge support from the ERC via the Starting Grant SQGT 101116964. During the first phase of this project, EF was employed at the University of Oxford and supported by the EPSRC programme grant [EP/S026347/1].
EF is grateful to Martin Hairer for discussions related to rough paths lifts of fBm with $H \leq \frac 14$.

\paragraph{Notation}

We write $X \lesssim Y$ to mean that there exists a constant $C > 0$ such that $X \leq CY$, and $X \asymp Y$ to mean that $X \lesssim Y$ and $Y \lesssim X$.
Subscripts to these symbols refer to dependence of the constant,
e.g.\ $\lesssim_s$ means $C$ may depend on $s$. For a path $X\colon [s,t]\to E$,
where $(E,\|\cdot\|)$ is a normed space, we write $\|X\|_{\Hol\gamma} = \sup_{s\leq u< v\leq t}|u-v|^{-\gamma}|X_u-X_v|$ for the $\gamma$-H\"older seminorm of $X$.
We write $C^{\Hol\gamma}$ for the space of $\gamma$-H\"older paths.
Unless otherwise specified, paths take values in $\R$,
e.g.\ $C[0,1]$ means the space of continuous functions from $[0,1]$ to $\R$.

\section{Area processes and sewing maps}\label{sec:determ}

For a path $X$ taking values in a vector space, write $X_{s,t} \coloneqq X_t - X_s$. Let 
\begin{equation}\label{eq:triangle_def}
\triangle[s,t] \coloneqq \{(u,v) \in [s,t]^2 \mid u \leq v\},\quad \triangle \coloneqq \triangle[0,1] .
\end{equation}
For $\gamma \in (0,1]$
denote
\(
C^\gamma = \CHol{\gamma}[0,1]
\).
Note that $C^\gamma \times C^\gamma$ can be identified with $\CHol{\gamma}([0,1],\bbR^2)$.

We further denote
\begin{equation}\label{eq:B_def}
B = \Big\{X\in \R^{\triangle} \, : \, \sup_{(s,t)\in\triangle}|X_{s,t}| < \infty\Big\}\;.
\end{equation}

\begin{defn}[Area process]\label{def:area}
Let $\gamma\in (0,1]$ and $(X,Y) \in C^\gamma \times C^\gamma$.
An \emph{area process} on $(X,Y)$ is an element
\begin{equation*}
\A(X,Y) = (\A_{s,t}(X,Y))_{(s,t) \in \triangle} \in B
\end{equation*}
such that, for all $0 \leq s \leq u \leq t \leq 1$, one has the \emph{Chen identity}
\begin{equation}\label{eq:chen}
\A_{s,t}(X,Y) = \A_{s,u}(X,Y) + \A_{u,t}(X,Y) + \frac12(X_{s,u}Y_{u,t} - Y_{s,u}X_{u,t})\;.
\end{equation}
For a subset $P \subset C^\gamma \times C^\gamma$, an area process on $P$ is a map $\A \colon P \to B$ such that $\A(X,Y)$ is an area process on $(X,Y)$ for all $(X,Y) \in P$.
\end{defn}

\begin{rem}[Link with rough paths]
A level-$2$ $\gamma$-H\"older weakly geometric rough path over $(X^1,X^2)$is a 
$2$-parameter function $\X^{ij}_{s,t}$ for $(s,t)\in\triangle$
and $i,j=1,2$ such that $\X^{ij}_{s,t} = \X^{ij}_{s,u} + \X^{ij}_{u,t} + X^i_{s,u}X^j_{u,t}$, $\X^{ij}_{s,t} + \X^{ji}_{s,t}=X^{i}_{s,t}X^j_{s,t}$
and the regularity condition $|\X^{ij}_{s,t}|\lesssim |t-s|^{2\gamma}$ holds uniformly in $(s,t)\in\triangle$.
Every such $\X$ defines an area process by $\X^{12}_{s,t} - \X^{21}_{s,t} = 2\A_{s,t}(X^1,X^2)$.
Conversely, an area process defines a level-$2$ weakly geometric rough path without the regularity condition.
\end{rem}

Given two area processes $\A$ and $\bar\A$ on $C^\gamma \times C^\gamma$, the Chen identity immediately implies that $\A_{s,t}(X,Y) - \bar \A_{s,t}(X,Y)$ is the increment of a path. Therefore, given one area process on $C^\gamma \times C^\gamma$, the space of area processes on $C^\gamma \times C^\gamma$ is parametrised by functions from $C^\gamma \times C^\gamma$ to $\{X\in \R^{[0,1]} \,:\, X_0=0 ,\|X\|_\infty<\infty\}$.

If $\gamma > \frac 12$ an area process on $C^\gamma \times C^\gamma$ is given by the L{\'e}vy area
arising from Young integration \cite{You36} by
\[
\A_{s,t}(X,Y) = \frac12\Big(\int_s^t X_{s,u} \dif Y_u - \int_s^t Y_{s,u} \dif X_u
\Big)
\;.
\]
This choice of $\A$ is furthermore the only one satisfying
the natural regularity condition $|\A_{s,t}(X,Y)| \lesssim |t-s|^{2\gamma}$ uniformly in $(s,t)\in\triangle$
because any other choice must differ from $\A_{s,t}(X,Y)$ by a $2\gamma$-H\"older continuous path, and $2\gamma > 1$. For $\gamma \leq \frac 12$, area processes are not unique, even under the regularity condition $|\A_{s,t}(X,Y)| \lesssim |t-s|^{2\gamma}$.

Given $Z \in C^\gamma\times C^\gamma$, define its local increment process on $[s,t]\subset[0,1]$ by
\begin{equation}\label{eq:r}
\begin{split}
&Z_{s, [s,t]} \coloneqq ([s,t] \ni u \mapsto Z_{s,u}) \in \CHol{\gamma}([s,t]) \times \CHol{\gamma}([s,t])
\;. 
\end{split}
\end{equation}
The following definition (cf.\ \cite[Remark 2.10]{BZ22_sewing}), which constitutes the central theme of this paper, formalises the intuition that $\A_{s,t}(X,Y)$ is only allowed to depend on increments of $X$ and $Y$ between $s$ and $t$.

\begin{defn}[Locality]\label{def:local}
Let $\A$ be an area process on $P\subset C^\gamma \times C^\gamma$.
We will say that $\A$ is \emph{local} if for $Z, Z' \in P$ and any $(s,t) \in \triangle$, one has the implication
\begin{equation*}
Z_{s,[s,t]} = Z'_{s,[s,t]} \implies \A_{s,t}(Z) = \A_{s,t}(Z') \;.
\end{equation*}
\end{defn}
Locality is equivalent to it being possible to view $\A_{s,t}$ as a map defined on $\{Z_{s,[s,t]} \mid Z \in P\}$ for all $(s,t) \in \triangle$.

\begin{defn}\label{def:homog}
We say that an area process $\A$ on $P\subset C^\gamma \times C^\gamma$ is \emph{homogeneous}
if for all $a,b\geq 0$ and $(X,Y)\in P$ such that $(aX,bY) \in P$, one has for all $(s,t) \in \triangle$
\[
\A_{s,t}(aX,bY) = ab\A_{s,t}(X,Y)\;.
\]
\end{defn}

\begin{rem}
In proving non-existence of homogeneous, local area processes for $\gamma\leq 1/2$ (see \autoref{thm:main_determ} \ref{pt:local}),
we only apply \autoref{def:homog} with either $a=0$ or $b=0$.
\end{rem}

\begin{defn}
We say that an area process $\A$ on $Z\in C^\gamma \times C^\gamma$ is \emph{time translation--invariant} if for all $(s,t) \in \triangle$ and $\tau \geq 0$ such that $t + \tau \leq 1$ and $Z_{s + \tau,u + \tau} = Z_{s,u}$ for all $u \in [s,t]$, one has
\[
\A_{s + \tau, t + \tau}(Z) = \A_{s, t}(Z)\;.
\]
\end{defn}

\begin{rem}\label{rem:timeloc}
Time translation-invariance should be viewed as a kind of locality which works within a single path.
Locality and time translation-invariance can be simultaneously strengthened
by requiring that, for all $Z,Z'$, if $Z'_{s + \tau,u + \tau} = Z_{s,u}$ for all $u \in [s,t]$, then
\[
\A_{s + \tau, t + \tau}(Z') = \A_{s, t}(Z)\;.
\]
We will call area processes satisfying this stronger condition \emph{strongly local}.
\end{rem}
	
Note that the Young integral, defined on $\bigcup_{\gamma > \frac 12} C^\gamma \times C^\gamma$ (in fact more generally $\bigcup_{\alpha + \beta > 1} C^\alpha \times C^\beta$) satisfies all the requirements defined up to now.

\begin{thm}\label{thm:main_determ}
\begin{enumerate*}[label=(\alph*)]
	\item \label{pt:local} There exists a 3-dimensional subspace $P\subset C^{1/2} \times C^{1/2}$ on which there is no local, homogeneous area process.
	\smallskip
	\\
	\item \label{pt:time} There exists a path $Z\in C^{1/2} \times C^{1/2}$ on which there is no time translation--invariant area process.
\end{enumerate*}
\end{thm}

\begin{rem}
The assumption that $\A$ takes values in $B\subset \R^\triangle$ is crucial for the non-existence results.
We show below in \autoref{cor:algebraic_area}
that there do exist area processes, with even larger domain, $\A \colon C[0,1] \times C[0,1] \to \R^\triangle$,
which satisfy Chen's identity \eqref{eq:chen}, are bilinear and strongly local in the sense of \autoref{rem:timeloc}.
\end{rem}

\begin{proof}[Proof of \autoref{thm:main_determ} \ref{pt:local}]
Let $P$ be the 3-dimensional subspace of $C^{1/2} \times C^{1/2}$ spanned by the path $Z$ given by \autoref{thm:axis} below and the two axis paths $X_t = (t,0)$ and $Y_t = (0,t)$.
By \autoref{lem:homog} below, every local, homogeneous area process vanishes on axis paths.
However, by \autoref{thm:axis}, there does not exist an area process on $Z$ which vanishes on axis paths.
\end{proof}

In the above proof, we used the following definition.

\begin{defn}\label{def:axis}
Let $e_1, e_2$ be the standard basis of $\R^2$.
We say that an area process $\A$ on $Z \in C^\gamma \times C^\gamma$ \emph{vanishes on axis paths} if for any interval $[u,v]\subset[0,1]$ such that, for all $t \in [u,v]$,
$Z_t = Z_u + (t-u)ae_i$ for some $i=1,2$ and $a \in \R$,
one has $\A_{u,v}(Z) = 0$.
\end{defn}

\begin{lem}\label{lem:homog}
Suppose $P\subset C^\gamma \times C^\gamma$ contains every axis path $X_t = (ta,0)$ and $Y_t = (0,ta)$ for $a \in \R$.
Then every local, homogeneous area process $\A$ on $P$
vanishes on axis paths.
\end{lem}

\begin{proof}
Let $\A$ be a homogeneous area process on $P$.
Then for an axis path $(X,Y)$ of the form $(X_t,Y_t) = (ta,0)$, where $a\in\R$, by homogeneity of $\A$,
\[
\A_{s,t}(X,Y) = \A_{s,t}(X,0\cdot Y) = 0
\;.
\]
The same applies to axis paths of the form $(X_t,Y_t) = (0,ta)$.
Then for a path $Z$ and interval $[u,v]$ as in \autoref{def:axis}, by locality of $\A$, one also has $\A_{u,v}(Z) = 0$.
\end{proof}

\begin{thm}\label{thm:axis}
There exists $Z \in C^{1/2}\times C^{1/2}$ on which there is no area process that vanishes on axis paths.
\end{thm}

\begin{proof}
Consider first any curve $(X,Y)$ that is piecewise linear on $s= u_0 < \ldots < u_n = t$ with paths that are axis-parallel line segments.
Then for any area process $\A$ on $(X,Y)$ that vanishes on axis paths,
it follows from Chen's identity that the area $\A_{s,t}(X,Y)$ is given classically by
\begin{equation}\label{eq:classical_area}
\A_{s,t}(X,Y) = \frac12\sum_{0 \leq i < j \leq n-1} (X_{u_i,u_{i+1}} Y_{u_j,u_{j+1}} - Y_{u_i,u_{i+1}} X_{u_j,u_{j+1}})
\;.
\end{equation}
For $n\geq 0$, let $Z^n=(X^n,Y^n)$ denote the $n$-th polygonal approximation to the L{\'e}vy C curve,
parametrised at constant speed. Concretely, $Z^n$ is given by a Lindenmayer system which works as follows: starting from the line segment from $(0,0)$ to $(1,0)$, recursively replace each line segment $\ell$ by the two equal sides of an isosceles right angled triangle for which $\ell$ is the hypotenuse.
The orientation of the triangle is chosen so that the first line segment makes an angle of $\pi/4$ with $\ell$.
See \autoref{fig:levyC} for the first few iterates.
Then every even iterate $Z^{2m}$ is piecewise linear with axis-parallel line segments.

\begin{figure}[h]
\centering

\begin{minipage}[t]{0.24\textwidth}\centering
	\includegraphics[width=\linewidth]{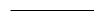}
\end{minipage}\hfill
\begin{minipage}[t]{0.24\textwidth}\centering
	\includegraphics[width=\linewidth]{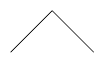}
\end{minipage}\hfill
\begin{minipage}[t]{0.24\textwidth}\centering
	\includegraphics[width=\linewidth]{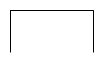}
\end{minipage}\hfill
\begin{minipage}[t]{0.24\textwidth}\centering
	\includegraphics[width=\linewidth]{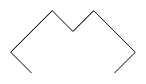}
\end{minipage}

\vspace{0.8ex}

\begin{minipage}[t]{0.24\textwidth}\centering
	\includegraphics[width=\linewidth]{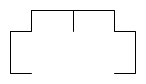}
\end{minipage}\hfill
\begin{minipage}[t]{0.24\textwidth}\centering
	\includegraphics[width=\linewidth]{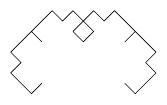}
\end{minipage}\hfill
\begin{minipage}[t]{0.24\textwidth}\centering
	\includegraphics[width=\linewidth]{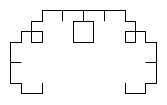}
\end{minipage}\hfill
\begin{minipage}[t]{0.24\textwidth}\centering
	\includegraphics[width=\linewidth]{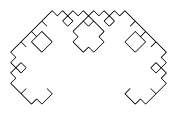}
\end{minipage}

\vspace{0.8ex}

\begin{minipage}[t]{0.24\textwidth}\centering
	\includegraphics[width=\linewidth]{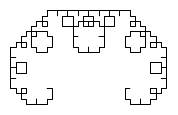}
\end{minipage}\hfill
\begin{minipage}[t]{0.24\textwidth}\centering
	\includegraphics[width=\linewidth]{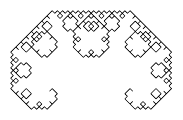}
\end{minipage}\hfill
\begin{minipage}[t]{0.24\textwidth}\centering
	\includegraphics[width=\linewidth]{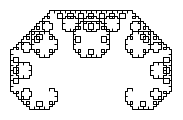}
\end{minipage}\hfill
\begin{minipage}[t]{0.24\textwidth}\centering
	\includegraphics[width=\linewidth]{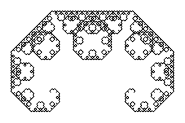}
\end{minipage}

\caption{The first 12 Lindenmayer iterates $Z^0, Z^1, \ldots, Z^{11}$ converging to the L{\'e}vy C curve.}
\label{fig:levyC}
\end{figure}

For each $m\geq 0$,
let $\widetilde Z^m$ denote the closed polygon obtained by following $Z^{2m}$ at constant speed from $(0,0)$ to $(1,0)$
over the time interval $[0,\frac12]$, and then following at constant speed the horizontal segment from $(1,0)$ back to $(0,0)$ over the time interval $[\frac12,1]$.

Finally, define the path $Z\colon [0,1]\to\R^2$ by setting $Z_0=(0,0)$ and, for $t\in [2^{-n-1},2^{-n}]$,
\[
Z_t
=
2^{-n/2} \widetilde Z^{2^{n}}_{2^{n+1}t-1}\;.
\]
We first claim that $Z\in C^{1/2}\times C^{1/2}$.
Indeed, since $\widetilde Z^{2^{n}}$ is a loop, $Z$ is continuous and vanishes at every dyadic point $2^{-n}$.
Moreover, it is easy to show that
\begin{equation*}
\sup_{m\geq 0}\|\widetilde Z^m\|_{\Hol{\frac12}}<\infty
\;.
\end{equation*}
Since the scaling $\lambda X_{\lambda^2 t}$ preserves the $1/2$-H\"older norm,
and since $Z$ vanishes at every dyadic point $2^{-n}$,
it readily follows that $Z\in C^{1/2}\times C^{1/2}$.

Suppose, towards a contradiction, that there exists an area process $\A$ on $Z$ that vanishes on axis paths.
We denote by $\A^{\cl}$ the classical signed area defined on piecewise axis-parallel paths by identity \eqref{eq:classical_area}.
Note that $\A^\cl$ is defined on $Z^{2n}$ and $\widetilde{Z}^{2n}$ for all $n\geq 0$.
We now define for even $n\geq 0$
\[
b_n = \A^\cl_{0,1}(Z^n)\;.
\]
Since $Z^0$ is horizontal, we have $b_0=0$.
We claim that, for even $n \geq 0$,
\begin{equation}\label{eq:two-step-levy}
b_{n+2}=b_n-\frac12
\;.
\end{equation}
Indeed, $Z^{n+2}$ is the concatenation of four copies of $Z^n$, each scaled by the factor $1/2$, and rotated respectively by
$
\frac{\pi}{2}, 0, 0,  -\frac{\pi}{2}$.
Since the classical area \eqref{eq:classical_area} is invariant under orthogonal transformations,
\[
\A^\cl_{j/4,(j+1)/4}(Z^{n+2})=\frac{b_n}{4}
\;,
\qquad j=0,1,2,3\;. 
\]
Since the signed area formed by the four line segments $(Z_0,Z^{n+2}_{1/4}), \ldots, (Z^{n+2}_{3/4},Z^{n+2}_{1})$ is $-1/2$,
we obtain by Chen's identity \eqref{eq:chen}
\[
b_{n+2}
=
\sum_{j=0}^3 \A^\cl_{j/4,(j+1)/4}(Z^{n+2})
-
\frac12 
=
b_n - \frac12\;,
\]
which proves \eqref{eq:two-step-levy}.
It follows by induction that, for all $m\geq 0$,
\begin{equation*}
b_{2m}=-\frac m2
\;.
\end{equation*}
Since adding the horizontal segment does not change the area, it follows that, for every $m\geq 0$,
\begin{equation}\label{eq:closed-even-levy-area}
\A^\cl_{0,1}(\widetilde Z^m)=\A^\cl_{0,1}(Z^{2m})=-\frac m2\;.
\end{equation}

Since each restriction $Z|_{[2^{-n-1},2^{-n}]}$ is piecewise linear with axis-parallel pieces, $\A_{2^{-n-1},2^{-n}}(Z)$ is given by the classical signed area \eqref{eq:classical_area}.
Since spatial scaling by $2^{-n/2}$ scales the area by $2^{-n}$,
it follows from \eqref{eq:closed-even-levy-area} that
\[
\A_{2^{-n-1},2^{-n}}(Z)
=
2^{-n}\A^\cl_{0,1}(\widetilde Z^{2^{n}})
=
-2^{-n}\frac{2^{n}}{2}
=
-1/2
\;.
\]
Since $Z|_{[2^{-n-1},2^{-n}]}$ is a closed loop, it follows from Chen's identity that, for every $n\geq 1$,
\[
\A_{2^{-n},1}(Z)
=
\sum_{k=0}^{n-1}\A_{2^{-k-1},2^{-k}}(Z)
=
-n/2
\;.
\]
This contradicts the fact that $\A(Z)\in B$, since elements of $B$ are bounded on $\triangle$.
\end{proof}

\begin{proof}[Proof of \autoref{thm:main_determ}\ref{pt:time}]
We construct a path $Z\in C^{1/2}\times C^{1/2}$ such that no area process on $Z$ can be time translation--invariant.

Let $C\colon[0,1]\to\R^2$ be the L{\'e}vy C curve, i.e. $C=\lim_{n\to\infty} Z^n$ where $Z^n$ is the $n$-th polygonal approximation to the L{\'e}vy C curve defined in the proof of \autoref{thm:axis}. Note that
$C_0=(0,0)$ and  $C_1=(1,0)$.
For $n\geq 1$, let $L^n\colon[0,2^{-n+1}]\to\R^2$ be the loop obtained by concatenating two paths:
on $[0,2^{-n}]$ we traverse the scaled copy
\[
t\mapsto 2^{-n/2}C(2^n t)\;,
\]
and on $[2^{-n},2^{-n+1}]$ we traverse its image under the rotation by $\pi$, translated so that the resulting path is a loop.
Thus $L^n$ has time-length $2^{-n+1}$ and spatial size $2^{-n/2}$.
(Contrary to the curves $\widetilde Z^n$ in the proof of \autoref{thm:axis}, we do not close $L^n$ with straight line segments.)

For $m\geq 0$, define the sequences
\[
N_m = 2^{2m+1}\;,\quad K_m = 2^{N_m-m-2}\geq 1\;.
\]
Since
\[
K_m\,2^{1-N_m}=2^{-m-1}\;,
\]
we may define a path $Z\colon[0,1]\to\R^2$ by setting $Z_0=0$ and, on each dyadic block $[2^{-m-1},2^{-m}]$
concatenating $K_m$ copies of the loop $L^{N_m}$.

Since each $L^{N_m}$ is a loop, $Z$ vanishes at the endpoints of all copies and in particular at every dyadic point $2^{-m}$.
Moreover, we claim that $Z\in C^{1/2}\times C^{1/2}$.
Indeed, on each half of $L^n$, the path is a translate of
\[
t\mapsto 2^{-n/2}C(2^n t)\;,
\]
which has the same $\frac12$-H\"older seminorm as $C$.
Since $Z$ vanishes on the endpoint of every copy of $L^{N_m}$, the $\frac12$-H\"older seminorm of $Z$ is bounded by a multiple of that of $C$. Thus $Z\in C^{1/2}\times C^{1/2}$.

Assume now, for contradiction, that there exists a time translation--invariant area process $\A$ on $Z$.
For even $n\geq 4$ and $r\in\Z/4\Z$, let $\cI_n^r$ denote the collection of intervals
$I=[s,s+2^{-n}]$ such that the restricted path
\[
u\mapsto Z_{s+u}-Z_s\;,\qquad u\in[0,2^{-n}]\;,
\]
coincides with
\[
u\mapsto 2^{-n/2}R_{r\pi/2}C(2^n u)\;,
\]
where $R_\theta$ is the rotation of the plane by $\theta \in \R/2\pi\Z$.
By construction, since $N_0=2$, for every even $n\geq 4$ and every $r\in\Z/4\Z$, the set $\cI_n^r$ is non-empty.

If $I,J\in\cI_n^r$, then the corresponding restricted paths are equal up to a time shift inside the path $Z$.
Therefore $\A_I(Z) = \A_J(Z)$ by translation-invariance of $\A$, and we denote
\[
a_n^r =\A_I(Z) \;.
\]

Fix an even $n\geq 4$ and $r \in \Z/4\Z$, and let $I\in\cI_n^r$.
By the two-step self-similarity of the L{\'e}vy C curve, the restriction of $Z$ to $I$ is the concatenation of four subintervals of length $2^{-n-2}$ belonging respectively to the classes
\[
r+1\;,\quad r\;,\quad r\;,\quad r-1 \;.
\]
Applying Chen's identity on $I$, we obtain
\begin{equation}\label{eq:an_recursion}
a_n^r
=
a_{n+2}^{r+1}+2a_{n+2}^{r}+a_{n+2}^{r-1}+2^{-n}c
\;,
\end{equation}
where $c=-1/2$ is the signed area of the two-step L{\'e}vy pattern (i.e. the signed area of the curve $Z^2$ in \autoref{fig:levyC}).
Now define
\[
b_n = \sum_{r\in\Z/4\Z} a_n^r\;.
\]
The recursion \eqref{eq:an_recursion}, with $c=-1/2$, implies that
\begin{equation}\label{eq:bn_recursion}
b_n = 4b_{n+2}-2^{-n+1}
\;,
\end{equation}
or equivalently that $B_n = 2^nb_n$ satisfies the recursion
$
B_n= B_{n+2}-2
$,
which implies $B_n = B_4 + n-4$, and therefore
\begin{equation}\label{eq:bn_explicit}
b_n = 2^{-n} (2^4 b_4 + n-4)
\;.
\end{equation}
Observe now that 
a copy of the loop $L^n$ is obtained by concatenating one interval of type $r=0$ and one interval of type $r=2$,
therefore by Chen's identity, the area of $L^n$ is
\[
a_n^0+a_n^2 = 2\sum_{r\in\Z/4\Z} a_{n+2}^r + 2^{-n+1}c
=
2b_{n+2} - 2^{-n}
=b_n/2
\;,
\]
where again $c=-1/2$ and we used \eqref{eq:bn_recursion} in the last equality.
For $m\geq 1$, since $N_m\geq 4$ and
$Z|_{[2^{-m-1},2^{-m}]}$ is $K_m = 2^{N_m-m-2}$ concatenations of the loop $L^{N_m}$,
we thus obtain by Chen's identity, using \eqref{eq:bn_explicit},
\begin{align*}
\A_{2^{-m-1},2^{-m}}(Z)
= K_m b_{N_m} / 2
&=
2^{N_m-m-3} \bigl(2^{-N_m} (2^4 b_4 + N_m-4 ) \bigr)
\\
&=2^{-m-3} (2^4 b_4 + N_m-4 )
\;.
\end{align*}
Since we have chosen $N_m=2^{2m+1}$, the final term is larger than $1$ for all $m \geq m_0$ (where $m_0$ depends only on $b_4$), which implies
that $\A_{2^{-m-1},2^{-m_0}}(Z) \geq m-m_0+1$,
which contradicts the fact that $\A(Z)\in B$.
\end{proof}

The preceding proofs use modified versions of the L{\'e}vy C curve and its polygonal approximations.
If one instead makes stronger assumptions on $\A$, namely homogeneity, and parametrisation-, rotation-, and spatial translation-invariance, then the unmodified L{\'e}vy C curve already gives a contradiction.
More precisely, suppose there exists a homogeneous area process $\A\colon C^{1/2}\times C^{1/2} \to B$
which is invariant under the addition of constants in the sense that $\A(a + Z) = \A(Z)$ for $a \in \R^2$, and parametrisation-invariant in the sense that, for every affine increasing bijection $\phi\colon [0,1]\to [s,t]$,
\[
	\A_{0,1}(Z_\phi)=\A_{s,t}(Z)
\]
where $Z_\phi = Z \circ \phi$. These properties imply strong locality in the sense of \autoref{rem:timeloc}: if $Z'_{s + \tau,u + \tau} = Z_{s,u}$ for all $u \in [s,t]$ then
\[
\A_{s,t}(Z) = \A_{0,1}(Z_\phi) = \A_{0,1}((Z_s - Z'_{s+\tau}) + Z'_{\phi(\cdot) +\tau}) = \A_{0,1}(Z'_{\phi(\cdot) +\tau}) = \A_{s+\tau, t+\tau}(Z').
\]
Assume further that $\A$ is rotation-invariant 
in the sense that $\A_{s,t}(R_\theta Z) = \A_{s,t}(Z)$, where we recall that $R_\theta$ is the rotation of the plane by $\theta \in \R/2\pi\Z$.
Let $C = (X,Y)\in C^{1/2}\times C^{1/2}$ be the L{\'e}vy C curve.
Then two self-similar halves of $C$ have the same area as $C$, scaled by the factor $\frac12$,
hence Chen's identity implies
\[
\begin{split}
	\A_{0,1}(C)
	&=
	\A_{0,\frac12}(C)
	+
	\A_{\frac12,1}(C)
	+
	\frac12
	\bigl(
	X_{0,\frac12} Y_{\frac12,1}
	-
	Y_{0,\frac12} X_{\frac12,1}
	\bigr)
	\\
	&=
	\frac12 \A_{0,1}(C)
	+
	\frac12 \A_{0,1}(C)
	-
	\frac14
	\;,
\end{split}
\]
which is a contradiction. More generally, for a self-similar curve with $m$ pieces and spatial dilation scale $\sqrt \lambda$, the same computation gives
\(
	(1-m\lambda)\A_{0,1}=a,
\)
where $a$ is the polygonal area contribution coming from Chen's identity; the L{\'e}vy C curve is the critical case $m\lambda=1$ with $a\neq0$. Other fractals could be constructed with $\lambda m > 1$, and for such fractals the optimal H\"older regularity (which is easily seen to be $\log(1/\lambda)/(2 \log m)$) would be less than $1/2$, but parametrisation invariance would not immediately yield a contradiction (but arguments similar to those of \autoref{thm:main_determ} still would).

\subsection{Non-existence of local sewing}
\label{subsec:sewing}

We now discuss consequences of \autoref{thm:main_determ} and \autoref{thm:axis} for sewing maps.
Existence of (linear) sewing maps in different contexts was shown in \cite{Gub04, Feyel06_sewing, BZ22_sewing}.
We give below an abstract definition of such maps and deduce non-existence results under minimal regularity assumptions.
The link between sewing maps and area processes (and thus rough path lifts)
is discussed after \autoref{def:weak_sewing} below.

Denote
\begin{equation}\label{eq:triangle_3_def}
\triangle^3[s,t]
=
\{(a,b,c)\in [s,t]^3 \,:\, a\leq b\leq c\}\;,
\qquad
\triangle^3
= \triangle^3[0,1]\;,
\end{equation}
and set
\[
\cC_2 = C(\triangle)\;,
\qquad
\cC_3 = C(\triangle^3)\;.
\]
For \(k=2,3\) and \(\gamma>0\), define
\[
\cC_k^\gamma
\coloneqq
\left\{
\Xi\in \cC_k
\,:\,
\|\Xi\|_\gamma<\infty
\right\}
\;,
\qquad
\|\Xi\|_\gamma
\coloneqq
\sup_{0\leq t_1 < \cdots < t_k\leq 1}
\frac{|\Xi_{t_1,\ldots,t_k}|}{(t_k-t_1)^\gamma}
\;.
\]
Define further
\[
\cC_2^{1,\log}
\coloneqq
\left\{
A\in \cC_2
\,:\,
\|A\|_{1,\log}<\infty
\right\}\;,
\qquad
\|A\|_{1,\log}
\coloneqq
\sup_{0\leq s<t\leq 1}
\frac{|A_{s,t}|}{(t-s)(1+|\log(t-s)|)}
\;.
\]
Define the linear map $\delta_2 \colon \R^{\triangle^2} \to \R^{\triangle^3}$ by
\begin{equation}\label{eq:delta_def}
(\delta_2 A)_{s,u,t}
=
A_{s,t}-A_{s,u}-A_{u,t}
\;.
\end{equation}
Note that $\delta_2$ maps $\cC_2$ into $\cC_3$.

\begin{defn}
A \emph{\(\gamma\)-sewing map} is a right inverse to \(\delta_2\) on $\cC_3^\gamma$,
namely a map
\(
\Lambda \colon
\delta_2(\cC_2)\cap \cC_3^\gamma
\to
\cC_2
\)
such that
\(
\delta_2\Lambda = \mathbbm 1
\)
with \(\Lambda\Xi\in \cC_2^\gamma\) for \(\gamma\neq 1\), and
\(\Lambda\Xi\in \cC_2^{1,\log}\) for \(\gamma=1\).
\end{defn}

\begin{rem}\label{rem:integration}
Denoting \(\cC_1 = C[0,1]\) and defining \(\delta_1\colon \cC_1 \to \cC_2\) by
\(
(\delta_1 f)_{s,t} = f_t - f_s
\),
the maps \(\cC_1 \xrightarrow{\delta_1} \cC_2 \xrightarrow{\delta_2} \cC_3 \cdots\)
are the first terms in an exact cochain complex, see \cite[\S 2.2]{Gub04}.
The sequence $(\cC_\bullet^\gamma, \delta_\bullet)$ is a sub-cochain complex. A linear $\gamma$-sewing map $\Lambda$ yields a linear \emph{integration map} $\cI \colon \delta_2^{-1}(\cC_3^\gamma) \to \cC_1$ by
\begin{equation}\label{eq:integration_def}
(\cI A)_0 = 0\;, \qquad \delta_1 \cI  A  = A - \Lambda \delta_2 A\;,
\end{equation}
which satisfies 
\(\cI \delta_1 f = f - f_0\) 
and similar analytic properties as $\Lambda$. Conversely, an integration map defines a $\gamma$-sewing map by the identity \eqref{eq:integration_def} applied to $\Xi=\delta_2A \in \delta_2(\cC_2) \cap \cC_3^\gamma$ (this definition does not depend on the choice of $A \in \delta_2^{-1}(\Xi)$).
The sewing/integration maps can be viewed as (partial) functions defining a splitting
\[
\begin{tikzcd}
\cC_1 \arrow[r, "\delta_1"]  &\cC_2 \arrow[r, "\delta_2"] \arrow[l, bend left, rightharpoonup, "\cI"] &\cC_3 \arrow[l, bend left, rightharpoonup, "\Lambda"]
\end{tikzcd}
\;.
\]
This link between sewing and integration maps is central in rough path theory \cite{Gub04,Feyel06_sewing, BZ22_sewing,FH20}.
We will exploit this link in the construction of strongly local sewing maps under different analytic conditions in the proof of \autoref{thm:algebraic_local_sewing}.
\end{rem}

For \(\gamma>1\), there exists a unique \(\gamma\)-sewing map; it is given by Riemann--Stieltjes--Young integration and is a continuous linear map, see \cite[Lemma~2.1]{Feyel06_sewing} or \cite[Lemma~4.2]{FH20}.
For \(\gamma\leq 1\), \cite[Theorem~3.2 and Proposition~3.4]{BZ22_sewing}, building on \cite{LV07, TaZa20}, construct continuous linear sewing maps, but their construction is non-local: the value of \(\Lambda\Xi\) on an interval \([s,t]\) may depend on values of \(\Xi\) outside \(\triangle^3[s,t]\).
It was left open, see \cite[Remark~2.10]{BZ22_sewing}, whether one can choose such a sewing map locally.

\autoref{thm:main_determ} and \autoref{thm:axis} give a negative answer to this question for all $\gamma \leq 1$ under the natural condition that $\Lambda(0) = 0$, see \autoref{cor:local_sewing}.
In fact, this and \autoref{cor:time_sewing} prove much stronger results:
any right-inverse \(\delta_2\Lambda=\mathbbm 1\) taking values in the space of bounded functions cannot vanish locally at zero or be time translation--invariant.
The following definition makes these notions precise.

\begin{defn}\label{def:weak_sewing}
Recall the set $B\supset \cC_2$ from \eqref{eq:B_def}.
A \emph{weak $\gamma$-sewing map} is a map $\Lambda \colon \delta_2(\cC_2) \cap \cC_3^\gamma \to B$ such that $\delta_2 \Lambda = \mathbbm 1$.
We say that
\begin{enumerate}[label=(\roman*)]
\item
$\Lambda$ is \emph{local} if for all $(s,t) \in \triangle$ and $\Xi, \Xi' \in \delta_2(\cC_2) \cap \cC_3^\gamma$ such that $\Xi|_{\triangle^3[s,t]} = \Xi'|_{\triangle^3[s,t]}$, it follows that $\Lambda \Xi|_{\triangle[s,t]} = \Lambda \Xi'|_{\triangle[s,t]}$,

\item
$\Lambda$ \emph{vanishes locally at zero}
if for all $\Xi \in \delta_2(\cC_2) \cap \cC_3^\gamma$ and $(s,t)\in\triangle$ such that $\Xi|_{\triangle^3[s,t]} = 0$,
one has $\Lambda \Xi|_{\triangle[s,t]} = 0$,

\item \label{pt:time_translation_def}$\Lambda$ is \emph{time translation--invariant} if for all $(s,t) \in \triangle$ and $\tau >0$ such that $t+\tau\leq 1$ and $\Xi \in \delta_2(\cC_2) \cap \cC_3^\gamma$
such that $\Xi_{a,b,c} = \Xi_{a+\tau,b+\tau,c+\tau}$ for all $(a,b,c)\in\triangle^3[s,t]$,
one has $(\Lambda \Xi)_{s,t} = (\Lambda \Xi)_{s+\tau,t+\tau}$.
\end{enumerate}
\end{defn}

Note that a $\gamma$-sewing map is always a weak $\gamma$-sewing map, but the converse is not true.
Given $\gamma \leq 1$ and $(X,Y) \in C^{\gamma/2} \times C^{\gamma/2}$, consider the \say{germ} defined by
\[
A \in \cC_2, \quad A_{s,t} = -\frac 12 (X_s Y_{s,t} - Y_s X_{s,t})\;.
\]
A direct computation shows that
\[
(\delta_2 A)_{s,u,t} =  \frac 12 (X_{s,u}Y_{u,t} - Y_{s,u}X_{u,t})
\]
and thus $\delta_2 A \in \cC^{\gamma}_3$.
A weak $\gamma$-sewing map then defines an area process in the sense of \autoref{def:area} by
\begin{equation}\label{eq:area_from_sewing}
\A_{s,t}(X,Y) =  (\Lambda \delta_2 A)_{s,t}
\;.
\end{equation}
Consider now $Z=(X,Y)$ which is an axis path on $[u,v]\subset[0,1]$, i.e. $Z_t = Z_u + (t-u)ae_i$ for all $t \in [u,v]$ for some $i=1,2$, $a \in \R$.
Then $\delta_2 A|_{\triangle^3[u,v]} = 0$.
Therefore, if $\Lambda$ vanishes locally at zero, then $\A_{u,v}(Z) = 0$
and thus $\A$ vanishes on axis paths in the sense of \autoref{def:axis}.
We thus obtain a corollary of \autoref{thm:axis}.

\begin{cor}\label{cor:local_sewing}
Let $\gamma\in (0,1]$.
There is no weak $\gamma$-sewing map which vanishes locally at zero.
In particular, there is no local weak $\gamma$-sewing map $\Lambda$ such that $\Lambda(0)=0$.
\end{cor}

By similar considerations, a time translation--invariant weak $\gamma$-sewing map would yield a time translation--invariant area process on $C^{\gamma/2}\times C^{\gamma/2}$.
We thus obtain the following corollary of \autoref{thm:main_determ} \ref{pt:time}.

\begin{cor}\label{cor:time_sewing}
Let $\gamma\in (0,1]$.
There is no weak $\gamma$-sewing map which is time translation--invariant.
\end{cor}

\subsection{Local unbounded sewing maps and area processes}

In the next \autoref{thm:algebraic_local_sewing}, we study the effects of enlarging the target space of $\Lambda$ from the space $B$ of bounded functions to all of \(\R^\triangle\)
as well as enlarging the domain of $\Lambda$.
With target space $\R^\triangle$, we show that there do exist strongly local (in particular local and time translation--invariant) sewing maps on the domain \(\delta_2\cC_2\), see \autoref{thm:algebraic_local_sewing}\ref{pt:cont_exist}.
Moreover, somewhat surprisingly,
on the larger domain \(\delta_2 \R^{\triangle}\), there also exist local sewing maps but \emph{not} time translation--invariant ones (\autoref{thm:algebraic_local_sewing}\ref{pt:discont_exist}-\ref{pt:discont_nonexist}).
We obtain in \autoref{cor:algebraic_area} the existence of (strongly) local area processes taking values in $\R^{\triangle}$ and with domain $C[0,1]\times C[0,1]$ or $\R^{[0,1]}\times \R^{[0,1]}$. These results are non-constructive and make use of the axiom of choice.

Denote
\[
\cC_2[s,t] = C(\triangle[s,t])\;,
\quad
\cC_3[s,t] = C(\triangle^3[s,t])\;,
\]
where we recall the notation from \eqref{eq:triangle_def} and \eqref{eq:triangle_3_def}.
Define $\delta_2\colon \R^{\triangle[s,t]} \to \R^{\triangle^3[s,t]}$ by \eqref{eq:delta_def} and denote
\[
E[s,t] = \delta_2\cC_2[s,t]  \subset \cC_3[s,t]
\;,\qquad
F[s,t] = \delta_2 \R^{\triangle[s,t]} \subset \R^{\triangle^3[s,t]}
\;.
\]

\begin{thm}
\label{thm:algebraic_local_sewing}
\begin{enumerate*}[label=(\roman*)]
\item \label{pt:cont_exist} There exists a linear map \(\Lambda \colon E[0,1] \to \R^{\triangle}\)
such that $\delta_2\Lambda=\mathbbm 1$ and $\Lambda$ is strongly local (cf.\ \autoref{rem:timeloc}) in the sense that, for all \(\tau\geq 0\) and $(s,t) \in \triangle$ such that \(t+\tau\leq 1\) and \(\Xi,\Xi'\in E [0,1] \) such that $\Xi_{a,b,c} = \Xi'_{a+\tau,b+\tau,c+\tau}$ for all $(a,b,c)\in\triangle^3[s,t]$,
one has
\(
(\Lambda\Xi)_{s,t}=(\Lambda\Xi')_{s+\tau,t+\tau}
\).
\medskip
\\

\item \label{pt:discont_exist} There exists a linear map \(\Lambda \colon F[0,1] \to \R^{\triangle}
\)
such that $\delta_2\Lambda =\mathbbm 1$ 
and which is local in the sense that, for all $(s,t) \in \triangle$ and \(\Xi,\Xi'\in F [0,1] \) such that $\Xi|_{\triangle^3[s,t]}= \Xi'|_{\triangle^3[s,t]}$,
one has
\(
(\Lambda\Xi)_{s,t}=(\Lambda\Xi')_{s,t}
\).
\medskip
\\

\item \label{pt:discont_nonexist} There does not exist a map
\(
\Lambda \colon F[0,1] \to \R^{\triangle}
\)
such that 
\(
\delta_2\Lambda=\mathbbm 1
\)
and which is time translation--invariant in the sense of \autoref{def:weak_sewing}\ref{pt:time_translation_def}.

In fact, there exists \(\Xi\in F[0,1]\) which is translation-invariant in the sense that $\Xi_{s,u,t} = \Xi_{s+\tau,u+\tau,t+\tau}$ for all $(s,u,t),(s+\tau,u+\tau,t+\tau)\in\triangle^3$
and for which there exists no \(A\in \R^{\triangle}\) with $\delta_2A=\Xi$ and which is translation-invariant in the sense that \(A_{s,t} = A_{s+\tau,t+\tau}\) for all $(s,t),(s+\tau,t+\tau)\in\triangle$.
\end{enumerate*}
\end{thm}

\begin{rem}
The assumption that the domain of $\Lambda$ is $\delta_2\cC_2$ in \ref{pt:cont_exist} is stronger than necessary.
The proof below reveals that $\cC_2$ can be replaced by bounded measurable functions that are suitably \say{regulated} on the diagonal.
\end{rem}

\begin{proof}
\ref{pt:cont_exist} We follow the idea of \autoref{rem:integration}
and construct first a \say{strongly local integration map}.
Let \(\mathcal V_{c}\subset\R^{(0,1]}\) be the subspace of all functions
\(f\colon(0,1]\to\R\) for which
\(
\lim_{\eps\downarrow0}f(\eps)
\)
exists.
This limit defines a linear functional on
\(\mathcal V_{c}\).
Extending a basis of
\(\mathcal V_{c}\) to a basis of \(\R^{(0,1]}\), which is possible by the axiom of choice, we obtain
a linear functional
\[
\Lim_{\eps\downarrow0}\colon
\R^{(0,1]}\to\R
\]
such that
\[
\Lim_{\eps\downarrow0} f(\eps)
=
\lim_{\eps\downarrow0} f(\eps)
\]
whenever the ordinary limit on the right-hand side exists.

For \(0<L\leq1\) and \(A\in \cC_2[0,L]\), define for
\(0<\eps<L\) the \say{approximate integral}
\begin{equation}\label{eq:J_def}
J_L^\eps(A)
=
\frac1\eps
\int_0^{L-\eps}
(A_{r,r+\eps}-A_{r,r}) \dif r
\;.
\end{equation}
For \(\eps\geq L\), define \(J_L^\eps(A)=0\). Now define the \say{integration map} \(I_L\colon \cC_2[0,L] \to\R\) by
\[
I_L(A)
=
\operatorname{Lim}_{\eps\downarrow0}
J_L^\eps(A)\;.
\]
For $L=0$, we also define $I_0\colon \cC_2[0,0] \to \R$ by $I_0 = 0$.
Since \(J_L^\eps\colon \cC_2[0,L] \to\R\) is linear, \(I_L\colon \cC_2[0,L] \to\R\) is also linear.

We now prove two properties of \(I_L\).
First, for \(A\in \cC_2[0,L]\) and \(0\leq l\leq L\), define \(A^{0,l} \in \cC_2[0,l]\) and \(A^{l,L}\in \cC_2[0,L-l]\) by
\begin{align*}
A^{0,l}_{s,t}&=A_{s,t},
\qquad
0\leq s\leq t\leq l\;,
\\
A^{l,L}_{s,t}&=A_{l+s,l+t},
\qquad
0\leq s\leq t\leq L-l\;.
\end{align*}
We claim that \(I_L\) is additive in the sense that
\begin{equation}\label{eq:additivity}
I_L(A)
=
I_l(A^{0,l})
+
I_{L-l}(A^{l,L})\;.
\end{equation}
Indeed, for $l=0$ or $l=L$, the claim follows from $I_0 = 0$, so assume that \(0<l<L\).
Taking \(0<\varepsilon<\min(l,L-l)\), splitting the integral in \eqref{eq:J_def} into the intervals \([0,l-\eps]\), \([l-\eps,l]\) and \([l,L-\eps]\),
and remarking that the final integral is precisely \(J_{L-l}^\varepsilon(A^{l,L})\),
we obtain
\[
J_L^\varepsilon(A)
=
J_l^\varepsilon(A^{0,l})
+
J_{L-l}^\varepsilon(A^{l,L})
+
R_\varepsilon\;,
\]
where
\[
R_\varepsilon
=
\frac1\varepsilon
\int_{l-\varepsilon}^{l}
(A_{r,r+\varepsilon}-A_{r,r}) \dif r.
\]
Note that \(|R_\varepsilon|
\leq
\sup_{r\in[l-\varepsilon,l]}|A_{r,r+\varepsilon}-A_{r,r}|\) and that the right-hand side converges to zero as \(\varepsilon\downarrow0\)
because \(A\) is continuous at \((l,l)\).
Applying \(\operatorname{Lim}_{\varepsilon\downarrow0}\) proves the claim \eqref{eq:additivity}.

Second, for \(L \in [0,1]\), we claim that if \(K\in \cC_2[0,L]\) satisfies
\(
\delta_2K=0
\),
then
\begin{equation}\label{eq:integral}
I_L(K)=K_{0,L}\;.
\end{equation}
To see this, define \(g_t=K_{0,t}\) which by \(\delta_2K=0\) satisfies
\(
K_{s,t}=g_t-g_s
\)
for \(0\leq s\leq t\leq L\).
Note that, since we are working on the closed simplex, \(\delta_2 K_{r,r,r} = -K_{r,r} = 0\)
for every \(r\in[0,L]\), and in particular \(g_0=K_{0,0}=0\).
This also shows \eqref{eq:integral} for $L=0$ since $I_0(K)=0=K_{0,0}$.
On the other hand, if $L>0$, then
\begin{equation}\label{eq:J_L_K}
J_L^\eps(K)
=
\frac1\eps
\int_0^{L-\eps}
K_{r,r+\eps} \dif r
=
\frac1\eps
\int_0^{L-\eps}
(g_{r+\eps} - g_r)\dif r
=
\frac1\eps
\Big(
\int_{L-\eps}^{L}g_r \dif r
-
\int_0^\eps g_r \dif r
\Big)\;.
\end{equation}
Since \(K\) is continuous, so is \(g\) and thus
\[
\lim_{\eps\downarrow0} J_L^\eps(K) = g_L-g_0=K_{0,L}\;.
\]
Since the limit as \(\eps\downarrow0\) of the left-hand side of \eqref{eq:J_L_K} is also equal to \(I_L(K)\), we obtain \eqref{eq:integral}.

We now proceed to define a sewing map on \([0,L]\) for \(L \in [0,1]\).
For \(\xi\in E[0,L]\),
choose any \(A\in \cC_2[0,L]\) such that
\(
\delta_2A=\xi
\)
and define
\[
\lambda_L(\xi)
=
A_{0,L}-I_L(A)\;.
\]
We first claim that \(\lambda_L(\xi)\) is independent of the choice of \(A\). Indeed, if \(\delta_2A'=\xi\), then \(K\coloneqq A-A'\) satisfies \(\delta_2K=0\), hence, by \eqref{eq:integral},
\[
I_L(K)=K_{0,L}\;.
\]
Therefore, by linearity of \(I_L\) and writing \(A=A'+K\),
\[
A_{0,L}-I_L(A)
=
A'_{0,L}+K_{0,L}-I_L(A') - I_L(K)
=
A'_{0,L}-I_L(A').
\]
Thus
\(
\lambda_L\colon E[0,L]\to\R
\)
is a well-defined linear map.

Next, we claim that \(\lambda_L\) satisfies the sewing identity
\begin{equation}\label{eq:sewing_lambda}
\lambda_L(\xi)
=
\lambda_l(\xi^{0,l})
+
\lambda_{L-l}(\xi^{l,L})
+
\xi_{0,l,L}
\end{equation}
for all \(0\leq l\leq L\). Indeed, if \(\xi=\delta_2A\), then \(\xi^{0,l}=\delta_2A^{0,l}\) and
\(\xi^{l,L}=\delta_2A^{l,L}
\), and thus
\begin{align*}
\lambda_L(\xi)
&=
A_{0,L}-I_L(A)\\
&=
A_{0,L}
-
I_l(A^{0,l})
-
I_{L-l}(A^{l,L}) \\
&=
\big(A_{0,l}-I_l(A^{0,l})\big)
+
\big(A_{l,L}-I_{L-l}(A^{l,L})\big)
+
A_{0,L}-A_{0,l}-A_{l,L} \\
&=
\lambda_l(\xi^{0,l})
+
\lambda_{L-l}(\xi^{l,L})
+
\xi_{0,l,L}\;,
\end{align*}
where we used the additivity property \eqref{eq:additivity} of \(I_L\) in the second equality.

We now finally define \(\Lambda\). Let \(\Xi\in E[0,1]\) and \(0\leq s\leq t\leq1\).
Set \(L=t-s\) and define the shift and restriction $\theta_{s,t}\Xi\in E[0,t-s]$ by
\[
(\theta_{s,t}\Xi)_{a,b,c}
=
\Xi_{s+a,s+b,s+c}\;,
\qquad
0\leq a\leq b\leq c\leq t-s\;.
\]
We then define
\[
(\Lambda\Xi)_{s,t}
=
\lambda_{t-s}(\theta_{s,t}\Xi)\;.
\]

The map \(\Xi\mapsto (\Lambda\Xi)_{s,t}\) is linear for every $s,t$ because \(\lambda_L\) is linear.
We now verify that \(\Lambda\) is a right-inverse of \(\delta_2\).
Consider \(0\leq s\leq u\leq t\leq1\) and set
\[
L=t-s\;,
\qquad
l=u-s\;,
\qquad
\xi=\theta_{s,t}\Xi\;.
\]
Then
\[
\xi_{0,l,L}
=
\Xi_{s,u,t}\;,
\qquad
\xi^{0,l} = \theta_{s,u}\Xi\;,
\qquad
\xi^{l,L} = \theta_{u,t}\Xi\;.
\]
By the sewing identity \eqref{eq:sewing_lambda} for \(\lambda_L\), one has \(
\lambda_L(\xi)
=
\lambda_l(\xi^{0,l})
+
\lambda_{L-l}(\xi^{l,L})
+
\xi_{0,l,L}
\),
which in terms of \(\Lambda\) becomes
\[
(\Lambda\Xi)_{s,t}
=
(\Lambda\Xi)_{s,u}
+
(\Lambda\Xi)_{u,t}
+
\Xi_{s,u,t}\;.
\]
This proves the right-inverse property \(\delta_2\Lambda\Xi=\Xi\).

Furthermore, \(\Lambda\) is strongly local by construction. Indeed, for $\Xi,\Xi'$, $\tau\geq 0$, and $(s,t)\in\triangle$ as in \ref{pt:cont_exist}, one has
\(
\theta_{s,t}\Xi=\theta_{s+\tau,t+\tau}\Xi'
\)
as elements of \(E[0,t-s]\), and therefore
\[
(\Lambda\Xi)_{s,t}
=
\lambda_{t-s}(\theta_{s,t}\Xi)
=
\lambda_{t-s}(\theta_{s+\tau,t+\tau}\Xi')
=
(\Lambda\Xi')_{s+\tau,t+\tau}\;.
\]

\ref{pt:discont_exist}
Define the real vector space
\[
G
=
\R \be
\oplus
\bigoplus_{[s,t]\subset[0,1]} F[s,t]
\;,
\]
where the direct sum is taken over all intervals \([s,t]\subset[0,1]\) with $s\leq t$ and \(\be\) is a formal symbol.
We treat $G$ as the free vector space generated by $\be$ and the formal symbols $[[s,t],
\xi]$ for $[s,t]\subset [0,1]$ and $\xi \in F[s,t]$, subject to the linearity relations $[[s,t],a\xi+b\eta]=a[[s,t],\xi]+b[[s,t],\eta]$.

For \((s,t)\in\triangle\), \(u\in[s,t]\), and \(\xi\in F[s,t]\), denote
\[
\xi^{s,u}=\xi|_{\triangle^3[s,u]}\;,
\qquad
\xi^{u,t}=\xi|_{\triangle^3[u,t]}\;,
\]
and let \(R\subset G\) be the subspace generated by all vectors of the form
\[
C_{[s,t]u,\xi} \coloneqq [[s,t],\xi]
-
[[s,u],\xi^{s,u}]
-
[[u,t],\xi^{u,t}]
-
\xi_{s,u,t}\be
\;.
\]
Denote the corresponding quotient map by
\[
q\colon G \to G/R\;.
\]
We claim first that $\be \notin R$ and thus $
q(\be)\neq 0$.
Suppose, towards a contradiction, that \(\be\in R\). Then
\begin{equation}\label{eq:be_def}
\be = \sum_{[s,t],u,\xi} a_{[s,t],u,\xi} C_{[s,t],u,\xi} \;,
\end{equation}
where the sums are finite and $a_{[s,t],u,\xi} \in \R$.
We define a linear functional \(S\) on the finite-dimensional span of all vectors \([[a,b],\xi]\) and $\be$ appearing in the right-hand side of \eqref{eq:be_def} after expanding the generators $C_{[s,t]u,\xi}$, and we will show that $S$ annihilates the right-hand side of \eqref{eq:be_def} but $S(\be)=1$, which is a contradiction.

Let $\Pi = (0=t_0 < t_1 < \cdots < t_n=1)$ be a finite partition of $[0,1]$ that contains all the time points $s,t,u$ appearing in \eqref{eq:be_def}.
For $s<t$ and $[[s,t],\xi]$ appearing in the sum \eqref{eq:be_def},
let $s=s_0<\cdots < s_\ell=t$ denote the points in $\Pi$ that are contained in $[s,t]$ and define
\[
S([[s,t],\xi])
=
\sum_{i=1}^{\ell-1} \xi_{s_0,s_i,s_{i+1}}\;,
\]
with the convention that $S([[s,t],\xi])=0$
if $s_0=s$ and $s_1=t$.
For $[[s,s],\xi]$ appearing in \eqref{eq:be_def}, define
\begin{equation}\label{eq:S_def}
S([[s,s],\xi])=-\xi_{s,s,s}\;.
\end{equation}
If $u \in [s,t]$ also appears in \eqref{eq:be_def}, then $u=s_j$ for some $j=0,\ldots,\ell$,
and it readily follows from $\xi \in \delta_2\R^{\triangle[s,t]}$ that
\[
S([[s,t],\xi])
=
S([[s,u],\xi^{s,u}])
+
S([[u,t],\xi^{u,t}])
+
\xi_{s,u,t}
\;.
\]
(For $u\in\{s,t\}$, this follows from \eqref{eq:S_def},
and for $u\in (s,t)$, this follows from writing $\xi_{a,b,c} = A_{a,c} - A_{a,b}-A_{b,c}$, from which we obtain
$S([[s,t],\xi]) = A_{s,t} - \sum_{i=0}^{\ell-1}A_{s_i,s_{i+1}}$.)
Extending $S$ to $\be$ by $S(\be)=1$ and extending to spans of vectors by linearity, we obtain
\(
S(C_{[s,t],u,\xi}) = 0
\).
Hence $S$ annihilates the right-hand side of \eqref{eq:be_def}, but $S(\be)=1$, which is a contradiction.
Therefore $\be\notin R$ as claimed.

Now define a linear functional on the one-dimensional subspace
$
\R q(\be)\subset G/R
$
by
\(
\ell(q(\be)) = 1
\).
Using the axiom of choice, we extend $\ell$ to a linear functional
\(
\ell\colon G/R\to\R\;
\)
that extends $\ell \colon \R q(\be) \to \R$.
For \(\Xi\in F[0,1]\) and $0\leq s \leq t \leq 1$, we then define
\[
(\Lambda\Xi)_{s,t}
=
\ell q([[s,t],\Xi|_{\triangle^3[s,t]}])\;.
\]
Then $\Lambda$ is linear in $\Xi$ and local by construction.

We now check that \(\Lambda\) is a right inverse to \(\delta_2\). Let
\(s\leq u\leq t\). Then since $C_{[s,t],u,\Xi|_{\triangle^3[s,t]}} \in R$, we have
\[
q([[s,t],\Xi|_{\triangle^3[s,t]}])
=
q([[s,u],\Xi|_{\triangle^3[s,u]}])
+
q([[u,t],\Xi|_{\triangle^3[u,t]}])
+
\Xi_{s,u,t}q(\be)
\;.
\]
Applying \(\ell\), and using \(\ell(q(\be))=1\), gives
\(
(\Lambda\Xi)_{s,t}
=
(\Lambda\Xi)_{s,u}
+
(\Lambda\Xi)_{u,t}
+
\Xi_{s,u,t}
\),
which is equivalent to $\delta_2 \Lambda \Xi = \Xi$.

\ref{pt:discont_nonexist}
Choose \(a,b>0\) such that
\begin{equation}\label{eq:L_def}
L \coloneqq a+b\leq 1
\end{equation}
and such that \(a\) and \(b\) are linearly independent over \(\Q\) (e.g.\ one is rational, the other irrational).
Using the axiom of choice, extend \(\{a,b\}\) to a Hamel basis \(\cB\) of \(\R\) over
\(\Q\).
Define an alternating \(\Q\)-bilinear map
\(
\omega\colon \R\times \R\to \R
\)
as follows.
Set
\begin{equation}\label{eq:omega_def}
\omega(a,b)=1\;,
\qquad
\omega(b,a)=-1\;,
\end{equation}
and, for every other ordered pair \((e,e')\in \cB\times\cB\), set \(\omega(e,e')=0\).
We then extend $\omega$ \(\Q\)-bilinearly to
\(\R\times\R\). Note that, since \(\omega\) is alternating, we have
\(\omega(u,u)=0\).
Define \(D\in\R^{\triangle}\) by 
\(D_{s,t}=-\omega(s,t)\)
and set
\(\Xi=\delta_2 D \in F[0,1]\).
For \(0\leq s\leq u\leq t\leq1\), by \(\Q\)-bilinearity of $\omega$, 
\begin{equation}\label{eq:Xi_formula}
\Xi_{s,u,t}
=
D_{s,t}-D_{s,u}-D_{u,t}
=
-\omega(s,t)+\omega(s,u)+\omega(u,t)
=
\omega(u-s,t-u)\;.
\end{equation}
It follows immediately that \(\Xi\) is translation-invariant in the sense of the proposition statement.

Suppose next for a contradiction that there exists \(A\in \R^{\triangle}\) such that \(\delta_2A=\Xi\) and such that \(A\) is translation-invariant.
Then the function
\(f\colon[0,1]\to\R\) defined by \(f(h) = A_{0,h}\)
satisfies
\(A_{s,t}=f(t-s)\).
Recalling \(a+b=L\) from \eqref{eq:L_def}, we obtain
\begin{equation}\label{eq:aL}
f(L)-f(a)-f(b)
=
A_{0,L}-A_{0,a}-A_{a,L}
=
\Xi_{0,a,L}
=
\omega(a,L-a)
=
\omega(a,b)=1
\;,
\end{equation}
where we used \(\delta_2 A=\Xi\) in the second equality,
\eqref{eq:Xi_formula} in the third equality,
and \eqref{eq:omega_def} in the final equality.
On the other hand, by the same reasoning,
\[
f(L)-f(b)-f(a)
=
A_{0,L}-A_{0,b}-A_{b,L}
=
\Xi_{0,b,L}
=
\omega(b,L-b)
=
\omega(b,a)
=
-1\;,
\]
which is a contradiction to \eqref{eq:aL}.
\end{proof}

\begin{cor}\label{cor:algebraic_area}
There exists a bilinear map
\[
\A \colon C([0,1])\times C([0,1]) \to \R^{\triangle}
\]
which satisfies Chen's identity \eqref{eq:chen} and is strongly local in the sense of \autoref{rem:timeloc},
i.e. if
$Z'_{s + \tau,u + \tau} = Z_{s,u}$ for all $u \in [s,t]$ then
$
\A_{s + \tau, t + \tau}(Z') = \A_{s, t}(Z)
$.
The same statement holds with the domain of $\A$ replaced by $\R^{[0,1]}\times \R^{[0,1]}$ and with \say{strongly local} replaced by \say{local} in the sense of \autoref{def:local}.
\end{cor}

\begin{proof}
This follows from defining $\A$ by \eqref{eq:area_from_sewing} for $\Lambda$ as in \autoref{thm:algebraic_local_sewing}\ref{pt:cont_exist}-\ref{pt:discont_exist}.
\end{proof}

\section{Local square-integrable lifts of fractional Brownian motion} \label{sec:fbm}

Let $X$ be a $d$-dimensional fractional Brownian motion (fBm) with Hurst parameter $H \leq 1/2$, that is the $\mathbb R^d$-valued Gaussian process with i.i.d. components, each with covariance function given by 
\[
R(s,t) \coloneqq \mathbb E X_s^\alpha X_t^\alpha = \frac 12 \big(s^{2H} + t^{2H} - (t-s)^{2H} \big), \qquad 0 \leq s \leq t \leq 1, \quad \alpha = 1,\ldots,d.
\]
We briefly recall the framework of isonormal Gaussian processes \cite{Nua06, NouPec12} applied to the specific case of $d$-dimensional fBm, with a focus on locality. We view $X_t(\omega) = \omega(t)$ as the canonical process defined on the Banach space
\[
\Omega \coloneqq \Omega_{0,1}\;,
\quad
\text{where } \; \Omega_{s,t} \coloneqq C_0([s,t], \mathbb R^d)
\coloneqq \{\omega \in C([s,t], \mathbb R^d) \,:\, \omega_s = 0\}
\]
and we equip $\Omega$ with its Borel sigma-algebra $\mathcal F$ and a Gaussian measure $\P$ under which $X$ is a fBm.
For $\alpha_1,\ldots,\alpha_n \in [d] \coloneqq \{1,\ldots,d\}$ we let $\{\alpha_1, \ldots, \alpha_n\}$ denote the set of distinct indices so that $\mathbb R^{\{\alpha_1,\ldots, \alpha_n\}} \cong \mathbb R^k$ for some $k \leq n$. Let 
\[
\Omega_{s,t}^{\alpha_1 \ldots \alpha_n} = C_0([s,t], \mathbb R^{\{\alpha_1, \ldots, \alpha_n\}}),
\]
endowed with its Borel sigma-algebra $\mathcal F^{\alpha_1,\ldots,\alpha_n}_{s,t}$. We denote, using the notation \eqref{eq:r},
\[
\mathscr r_{s,t} \colon \Omega \twoheadrightarrow \Omega_{s,t}, \quad \omega \mapsto \omega_{s,[s,t]},
\]
and note that this respects indices, i.e.\ induces maps $\Omega^{\alpha_1 \ldots \alpha_n} \twoheadrightarrow \Omega_{s,t}^{\alpha_1 \ldots \alpha_n}$.
We equip each $\Omega_{s,t}$ with the pushforward Gaussian measure $\P \circ\mathscr r_{s,t}^{-1}$.

Since the material in this section will not only establish a non-existence result but also a classification result,
we need a notion of local rough path lift which allows for $d$ components and degrees higher than level $2$---when $H \in (1/4, 1/3]$ we need level $3$---and which requires measurability. In particular, the notion of locality is not only w.r.t.\ time but also spatial indices.
We give the following definition in the context of fractional Brownian motion, but note that it makes sense in greater generality.

We refer to \cite{FV10} for background on rough paths.
All rough paths we consider are weakly geometric. We denote by $\mathcal G^n(\mathbb R^d) \subset T^n(\mathbb R^d) = \bigoplus_{k=0}^n (\R^d)^{\otimes k}$
the step-$n$ free nilpotent Lie group
(i.e. the group of truncated grouplike elements) equipped with a homogeneous metric (e.g.\ the Carnot--Carath\'eodory metric). For an element $A \in T^n(\mathbb R^d)$, we will use brackets in superscripts to denote its projections onto tensor powers of $\mathbb R^d$, i.e.\ $A = \sum_{k = 0}^n A^{(k)}$ with $A^{(k)} \in (\mathbb R^d)^{\otimes k}$.
For a word $\alpha_1 \ldots \alpha_n$ with $\alpha_i \in [d]$, we write $A^{\alpha_1 \ldots \alpha_n} \coloneqq \langle \alpha_1 \ldots \alpha_n, A \rangle$ for the dual pairing with the word.

\begin{defn}[Local rough path lift of fBm]\label{def:rplfbm}
Let $H \in (0, \frac 12]$ and $0 < \eta <H$ with $1/\eta < \lfloor 1/H \rfloor + 1$.
A \emph{local $\eta$-H\"older rough path lift} (\emph{local $\eta$-RPL} for short) of $X$ is a random variable $\bX \colon \Omega \to \CHol{\eta}([0,1], \mathcal G^{\lfloor 1/H \rfloor}(\mathbb R^d))$ such that $\bX^{(1)} = X$ and
for any $(s,t) \in \triangle$ and $\alpha_1,\ldots,\alpha_n \in [d]$ with $n \leq \lfloor 1/H \rfloor$, denoting $\bX_{s,t} \coloneqq \bX^{-1}_s\bX_t$, it holds that $\bX_{s,t}^{\alpha_1 \ldots \alpha_n}$ is $\mathscr r_{s,t}^{-1}(\mathcal F^{\alpha_1\ldots \alpha_n}_{s,t})$-measurable. 
We furthermore say that $\bX$ is square-integrable if $\bX_{s,t}^{\alpha_1 \ldots \alpha_n} \in L^2\Omega$ for all $(s,t) \in \triangle$ and $\alpha_1,\ldots,\alpha_n \in [d]$.
\end{defn}

A local $\eta$-RPL of $X$ can be regarded as a random variable defined on the probability space $\Omega_{s,t}$, endowed with the pushforward measure through $\mathscr r_{s,t}$, which respects indexing. Square-integrability of $\bX$ enables the use of Wiener chaos decompositions, which we proceed to introduce. For $(s,t) \in \triangle$ let $\cH^{(d)}_{s,t}$ be the Hilbert space
\begin{align*}
\cH^{(d)}_{s,t} &\coloneqq \cH^1_{s,t} \oplus \ldots \oplus \cH^d_{s,t},\\
\cH^\alpha_{s,t} &= \overline{\cE}{}^\alpha_{s,t}, \qquad \cE_{s,t}^\alpha = \mathrm{span}\{\mathbbm 1^\alpha_{[s,u]} \mid u \in [s,t] \},\\ \langle \mathbbm 1^\alpha_{[s,u]}, \mathbbm 1^\alpha_{[s,v]} \rangle &\coloneqq \mathbb E X^\alpha_{s,u}X^\alpha_{s,v} = R(s,s) + R(u,v) - R(s,v) - R(u,s)\;.
\end{align*}
where the closure is taken w.r.t.\ the inner product and the direct sum is orthogonal (equivalent to the components of the Gaussian process being independent). Since the components are all identically distributed, $\cH^1_{s,t} \cong \ldots \cong \cH^d_{s,t}$ canonically; when we wish to consider a single direct summand we will simply refer to $\cH_{s,t}$. As usual, we denote $\cH \coloneqq \cH_{0,1}$ (and similar).

Recall that multiple Wiener integration \cite[\S 2.7]{NouPec12} yields isometries
\begin{equation}\label{eq:isometry}
\sko^\bullet_{s,t} \colon \bigoplus_{m = 0}^\infty (\cH^{(d)}_{s,t})^{\odot m} \xrightarrow{\cong} L^2\Omega_{s,t}, \qquad \mathscr w^m_{s,t} \colon L^2\Omega_{s,t} \twoheadrightarrow \mathscr W^m_{s,t} \coloneqq  \sko^m_{s,t}((\cH^{(d)}_{s,t})^{\odot m}).
\end{equation}
We have used $\odot$ to denote the symmetric tensor product, i.e.\ for a Hilbert space $V$, $V^{\odot n} \subset V^{\otimes n}$ is the closure of the linear span of the symmetric elementary tensors 
\[
v_1 \odot \cdots \odot v_n \coloneqq \frac{1}{n!} \sum_{\sigma \in \mathfrak S_n} v_1 \otimes \cdots \otimes v_n.
\]
The subspace $\mathscr W^m_{s,t}\subset L^2\Omega_{s,t}$ is called the $m^\text{th}$ Wiener chaos over $[s,t]$, and $\mathscr w^m_{s,t}$ is the orthogonal projection onto it. Since all these operations commute with the natural inclusions, i.e.
\[
\begin{tikzcd}
(\cH^{(d)}_{s,t})^{\odot m} \arrow[d,hookrightarrow] \arrow[r,"\sko_{s,t}^m"] &\mathscr W^m_{s,t} \arrow[d,hookrightarrow] \\
(\cH^{(d)})^{\odot m} \arrow[r,"\sko^m"] &\mathscr W^m
\end{tikzcd}
\]
(as can be observed on the elementary tensors $\mathbbm 1^{\alpha_1}_{[s,u]} \odot \cdots \odot \mathbbm 1^{\alpha_n}_{[s,u]}$, which map to Hermite polynomials and have dense span in $(\cH^{(d)}_{s,t})^{\odot n}$), we will frequently omit the subscripts $s,t$. We also combine this with the notation for subsets of coordinates and the above considerations naturally restrict, i.e.\ 
\[
\cH^{\alpha_1\ldots\alpha_n}_{s,t} \coloneqq \bigoplus_{\beta \in \{\alpha_1, \ldots, \alpha_n\}} \cH^\beta_{s,t}, \qquad L^2\Omega^{\alpha_1\ldots\alpha_n}_{s,t} \cong \bigoplus_{m = 0}^\infty (\cH^{\alpha_1\ldots\alpha_n}_{s,t})^{\odot m}.
\]
It was shown in \cite{Jol07} that for $H \leq 1/2$, the Hilbert space $\cH_{s,t}$ can be identified with the space of functions (modulo a.e.\ identity),
\begin{equation}\label{eq:fracSobolev}
\begin{split}
	\cH_{s,t} &= \{ f \colon [s,t] \to \mathbb R \mid f^* \in H^{1/2 - H}(\mathbb R) \} , \qquad \text{with inner product} \\
	\langle f, g \rangle &= \frac{\Gamma(2H + 1)\sin(H\pi)}{2\pi} \int_{\mathbb R} \widehat{f^*}(x)\overline{\widehat{g^*}(x)} |x|^{1-2H} \dif x
\end{split}
\end{equation}
where $f^* \colon \mathbb R \to \mathbb R$ is the zero-extension of a function $f \colon [s,t] \to \mathbb R$, $\widehat{f^*} = \int_\mathbb{R} e^{itx} f^*(x) \dif t$ denotes Fourier transform, and $H^{1/2 - H}(\mathbb R)$ denotes the fractional Sobolev space.
To derive a time-domain representation of the norm, we introduce the Gagliardo seminorm for $I\subset\R$ and $\alpha \in (0,1)$
\[
[f]_{H^{\alpha}(I)}^2 \coloneqq \int_{I^2} \frac{|f(v) - f(u)|^2}{|v-u|^{1 + 2\alpha}} \dif u \dif v
\;.
\]
Then one has
\begin{equation}\label{eq:plancherel}
\begin{split}
\| f \|^2_{\cH_{s,t}} &\asymp_H [f^*]_{H^{1/2-H}(\R)}^2
= [f]_{H^{1/2-H}([s,t])}^2 + 2 \int_s^t |f(v)|^2 \int_{\mathbb R \setminus [s,t]} \frac{\dif u}{|v - u|^{2 - 2H}} \dif v \\
&\asymp_H [f]_{H^{1/2-H}([s,t])}^2 +  \int_s^t |f(v)|^2 ((v-s)^{2H-1} + (t-v)^{2H-1}) \dif v\;,
\end{split}
\end{equation}
where in the first bound we used \cite[Proposition 3.4]{hitchhiker}.

\begin{lem}\label{lem:characterization}
\(\| f \|^2_{\cH_{s,t}} \asymp_H [f]^2_{H^{1/2-H}([s,t])} + |t-s|^{2H-1}\|f\|_{L^2([s,t])}^2 \).
\end{lem}

\begin{proof}
This is a standard fact, but we give a proof for convenience.
The case \(H=1/2\) is obvious, so take \(H < 1/2\)
and denote \(\alpha=1/2-H\in (0,1/2)\).
The lower bound \(\| f \|^2_{\cH_{s,t}} \gtrsim_H [f]_{H^{\alpha}([s,t])}^2 + |t-s|^{-2\alpha}\|f\|_{L^2([s,t])}^2\)
follows immediately from \eqref{eq:plancherel}.
For the upper bound, by scaling and \eqref{eq:plancherel}, it suffices to consider \([s,t] = [0,1]\) and prove that
\begin{equation}\label{eq:hardy}
\int_0^1 |f(v)|^2 v^{-2\alpha} \dif v \lesssim_H [f]_{H^{\alpha}([0,1])}^2 + \|f\|_{L^2([0,1])}^2\;.
\end{equation}
Consider $\delta>0$ small.
The integral over the domain $v\in [\delta,1]$ is easy to bound.
For $v\in [0,\delta]$, we use $|f(v)|^2 \lesssim |f(u)|^2 + |f(v) - f(u)|^2$ and average over $u \in [v/2\delta,v/\delta]$ to obtain
\[
\int_0^{\delta} |f(v)|^2 v^{-2\alpha} \dif v \lesssim
\int_0^{\delta} v^{-2\alpha}
\delta v^{-1}\int_{v/2\delta}^{v/\delta} \{ |f(u)|^2  + |f(v) - f(u)|^2 \} \dif u
\dif v\;.
\]
The first double integral is bounded above by
\[
\delta \int_0^1\dif u |f(u)|^2 \int_{\delta u}^{2\delta u} v^{-1-2\alpha}\dif v \lesssim \delta^{1-2\alpha} \int_0^1 |f(u)|^2 u^{-2\alpha}\dif u\;,
\]
which can be absorbed into the left-hand side of \eqref{eq:hardy} for $\delta$ small enough because \(1-2\alpha>0\).
Next, since \(|v-u| \asymp_\delta v\) for \(u\in [v/2\delta,v/\delta]\), the second double integral is bounded by a multiple of \([f]_{H^{\alpha}([0,1])}^2\).
\end{proof}

\begin{lem}\label{lem:restriction}
For $(s,t) \in \triangle$ the maps $\cH_{s,t} \hookrightarrow \cH$ defined by zero extension is an isometry onto its image, the map $\cH \twoheadrightarrow \cH_{s,t}$ defined by restriction is continuous, and they are partial inverses to each other.
\end{lem}
\begin{proof}
The case $H = 1/2$ is obvious, so take $H < 1/2$. The claim regarding the map $\cH_{s,t} \hookrightarrow \cH$ follows directly by \eqref{eq:fracSobolev} and composing zero-extensions.
The second claim follows in turn from \autoref{lem:characterization}.
The final claim is obvious because zero-extension is a right inverse to restriction.
\end{proof}

Given the description of $\cH_{s,t}$ as a subspace of $L^2[s,t]$, we have a continuous embedding
\[
\cH^{(d)}_{s,t} \subset L^2([s,t], \mathbb R^d).
\]
Elements of $\cH$ belong to $\cH_{s,t}$ if and only if they are (represented by functions essentially) supported in $[s,t]$. Tensor products of $\cH_{s,t}$ are also continuously embedded into $L^2$ spaces:
\begin{equation}
(\cH_{s,t}^{\alpha_1 \ldots \alpha_n})^{\odot m} \subset L^2([s,t], \mathbb R^{\{ \alpha_1 ,\ldots , \alpha_n \}})^{\otimes m} = L^2([s,t]^m, (\mathbb R^{\{ \alpha_1 ,\ldots , \alpha_n \}})^{\otimes m})\;,
\end{equation}
and symmetry of the tensor product translates to $\mathfrak S_m$-equivariance, $f(u_{\sigma_1}, \ldots, u_{\sigma_m}) = \sigma_* f(u_1,\ldots, u_m)$, where $\sigma_*$ is the action on $m$-fold tensor powers defined analogously on elementary tensors. 

The next lemma will be needed to rule out certain rough path components having a first chaos projection.
We write
\[
C^\gamma_0([s,t])
\coloneqq
\{\omega\in \CHol{\gamma}([s,t]) \,:\, \omega_s=0\}, \qquad C^\gamma_0 \coloneqq C^\gamma_0([0,1]).
\]

\begin{lem}\label{lem:no1}
Let $h^0 \colon [0,1] \to \mathbb R$ be a deterministic path with $h^0_0=0$, $\overline h{}^1 \in \cH$, $h_t^1 \coloneqq \overline h{}^1\mathbbm 1_{[0,t]} \in \cH$, and assume that the (non-centred) Gaussian process
\[
Z_t = h^0_t + \sko(h^1_t), \qquad Z_0 = 0,
\]
has a.s.\ $\gamma$-H\"older sample paths with $\gamma > H$. Then $h^0 \in C^\gamma_0$ and $h^1 = 0$.
\end{lem}
\begin{proof}
Let $h^1_{s,t} \coloneqq \overline h{}^1|_{[s,t]}$ for any $(s,t) \in \triangle$.
By \autoref{lem:restriction}, $h^1_{s,t} \in \cH_{s,t}$ and $\overline h{}^1 \mathbbm 1_{[s,t]} \in \cH$. Since $Z$ defines a (non-centred) Gaussian measure on $C^\gamma_0$ with mean $h^0$, it follows that $h^0 \in C^\gamma_0$ (a $C^\gamma_0$-valued Bochner integral) and $\sko(h^1_\cdot)$ defines a centred Gaussian measure on $C^\gamma_0$. By Fernique's theorem, in particular $\| \sko(h^1_\cdot) \|_{\Hol\gamma} \in L^2$,
and therefore there exists $K \in \mathbb R$ such that for all $(s,t) \in \triangle$
\[
\| \sko(h^1_{s,t}) \|_{L^2} \leq K (t-s)^\gamma
\;.
\]
On the other hand, by \autoref{lem:characterization},
\begin{align*}
\| h^1 \|^2_{\cH_{s,t}} &\gtrsim (t-s)^{2H - 1} \| h^1 \|_{L^2[s,t]}^2.
\end{align*}
Putting these two together,
\[
(t-s)^{2H - 1} \| h^1 \|_{L^2[s,t]}^2 \lesssim_H \| h^1 \|^2_{\cH_{s,t}} =  \| \sko(h^1_{s,t}) \|_{L^2}^2 \lesssim_\gamma (t-s)^{2\gamma}.
\]
Therefore, for almost all $s \in [0,1]$, by the Lebesgue differentiation theorem and the fact that $\gamma > H$ we have
\[
h^1(s)^2 = \lim_{t \searrow s}\frac{1}{t-s}\int_s^t h^1(u)^2 \dif u \lesssim_{H,\gamma} \lim_{t \searrow s} (t-s)^{2(\gamma - H)} = 0. \qedhere
\]
\end{proof}

The following lemma is the main technical ingredient needed to exclude the existence of a local rough path lift for \(H\leq 1/4\).
\begin{lem}\label{lem:triangleH}
$\mathbbm 1_\triangle \in \cH^{\otimes 2}$ if and only if $H > 1/4$.
\end{lem}
\begin{proof}
By \autoref{lem:characterization} and the well-known Fourier series representation for the fractional Sobolev norm on \([0,1]\), we have
\[
\|f\|^2_{\cH^{\otimes 2}} \asymp \sum_{k\in\Z^2} |\scal{f,e^{i2\pi\scal{k,\cdot}}}_{L^2[0,1]^2}|^2 (|k_1|+1)^{1 - 2H}(|k_2|+1)^{1 - 2H}\;.
\]
Taking $f= \bb1_{\triangle[0,1]}$, a direct computation yields
\[
\scal{\bb1_{\triangle[0,1]},e^{i2\pi\scal{k,\cdot}}}_{L^2[0,1]^2} = \int_0^1 \int_0^y e^{i2\pi (k_1 x + k_2 y)}\dif x  \dif y 
=
\begin{dcases}
	\frac{1}{2}
	\quad &\text{if } k_1=k_2=0
	\\
	-\frac{1}{i2\pi k_1}
	\quad &\text{if } k_1\neq 0= k_2
	\\
	\frac{1}{i2\pi k_2}
	\quad &\text{if } k_1=0\neq k_2
	\\
	\frac{1}{i2\pi k_1}
	\quad &\text{if } k_1=-k_2\neq 0
	\\
	0 \quad &\text{otherwise.}
\end{dcases}
\]
We thus obtain that
\[
\|\bb1_{\triangle[0,1]}\|_{(\cH^1)^{\otimes 2}}^2
\asymp 
1 + \sum_{k_1 > 0} k_1^{-2} k_1^{1-2H} + \sum_{k_2 > 0} k_2^{-2} k_2^{1-2H} + \sum_{k_1 > 0} k_1^{-2} k_1^{2-4H}\;.
\]
The first two sums are finite for any $H>0$, while the last, $\sum_{k_1>0} k_1^{-4H}$, converges if and only if $H > \frac14$.
\end{proof}

The following embedding will instead be useful when we consider $H > \frac 14$.
\begin{lem}\label{lem:holder}
Let $H \in (0, 1/2)$ and $\gamma > 1/2 - H$.
Then the continuous inclusion $C^\gamma_0([s,t]) \subset \cH_{s,t}$ holds with
\[
\| f \|_{\cH_{s,t}} \lesssim_{H, \gamma} \| f \|_{\Hol\gamma} (t-s)^{\gamma + H}.
\]
For $H=1/2$, the same holds for any $\gamma \geq 0$.
\end{lem}
\begin{proof}
Consider first $H = \frac 12$, for which $\cH_{s,t} = L^2([s,t])$, and $f \in C^\gamma_0([s,t])$.
Using that $|f(v)| \leq \|f\|_{\Hol\gamma}(v-s)^\gamma$ (because $f(s) = 0$),
\begin{equation}\label{eq:L2case}
\| f \|_{L^2[s,t]}^2 = \int_s^t |f(v)|^2 \dif v \leq \| f \|_{\Hol\gamma}^2 \int_s^t (v-s)^{2\gamma} \dif v \lesssim_\gamma \| f \|_{\Hol\gamma}^2 (t-s)^{2\gamma + 1},
\end{equation}
and $2\gamma + 1 = 2(\gamma + H)$ for $H = \frac 12$. Assume now $H < \frac 12$.
Applying \autoref{lem:characterization},
the term \(|t-s|^{2H-1}\|f\|^2_{L^2[s,t]}\) is bounded above by a multiple of \(\| f \|_{\Hol\gamma}^2 (t-s)^{2(\gamma + H)}\) by \eqref{eq:L2case}.
For the first term \([f]^2_{H^{1/2-H}[s,t]}\), we have
\[
\int_{[s,t]^2} |v-u|^{2(\gamma + H) - 2} \dif u \dif v = (t-s)^{2(\gamma + H)} \int_{[0,1]^2} |y - x|^{2(\gamma + H) - 2} \dif x \dif y
\]
by the changes of variables $x = \frac{u-s}{t-s}$ and $y = \frac{v-s}{t-s}$ and the integral is finite since
\[
2(\gamma + H) - 2 > -1 \iff \gamma > 1/2 - H .
\]
Therefore
\[
[f]^2_{H^{1/2-H}[s,t]} \leq \| f \|_{\Hol\gamma}^2 \int_{[s,t]^2} |v-u|^{2(\gamma + H) - 2} \dif u \dif v \lesssim_{H,\gamma} \| f \|_{\Hol\gamma}^2 (t-s)^{2(\gamma + H)}\;.
\qedhere
\]
\end{proof}

\begin{thm}[Local $L^2$ rough path lifts of fBm]\label{thm:localfbm} Let $H \in (0,1/2]$ and $0 < \eta <H$ with $1/\eta < \lfloor 1/H \rfloor + 1$. Let $X$ be a $d$-dimensional $H$-fBm.
\begin{enumerate}[label=(\alph*)]
	\item \label{pt:H_less}If $d\geq 2$ and $H \leq 1/4$ there exists no local, square-integrable $\eta$-RPL of $X$.
	\item \label{pt:H_greater}
If $H > 1/4$, denoting $\overline{\bX}$ the canonical lift \cite{CQ02}, any local, square-integrable $\eta$-RPL $\bX$ must satisfy
\begin{equation}\label{eq:bX_level2}
	\bX^{(2)}_{s,t} = \overline{\bX}{}^{(2)}_{s,t} + \varphi^{(2)}_{s,t}
\end{equation}
and if $H \in (1/4, 1/3]$ additionally
\begin{equation}\label{eq:bX_level3}
	\bX^{\alpha\beta\gamma}_{s,t} = \overline{\bX}{}^{\alpha\beta\gamma}_{s,t} + \int_s^t \varphi^{\alpha\beta}_{s,u} \dif X^\gamma_u + \int_s^t \varphi^{\beta\gamma}_{u,t} \dif X^\alpha_u + \varphi^{\alpha\beta\gamma}_{s,t}
	\;,
\end{equation}
where the integrals are defined as Wiener integrals (e.g.\ $\int_s^t \varphi^{\alpha\beta}_{s,u} \dif X^\gamma_u = \sko(\varphi^{\alpha\beta}_{s,\cdot}\mathbbm 1_{[s,t]}^\gamma)$ and
\[
(0, \varphi^{(2)}, \varphi^{(3)})
\]
is a deterministic $\eta$-H\"older rough path above the zero path. Conversely, given any deterministic $\eta$-rough path $(0, \varphi^{(2)}, \varphi^{(3)})$ above the zero path, the formulas \eqref{eq:bX_level2}-\eqref{eq:bX_level3} define a local, square-integrable, geometric $\eta$-RPL of $X$. (Throughout this statement, we intend $\bX$ and $\varphi$ to be only considered up to level-$2$ if $\eta \in (\frac 13, \frac 12]$.)
\end{enumerate}
\end{thm}

\begin{rem}
	The formulas \eqref{eq:bX_level2}-\eqref{eq:bX_level3} for $\bX$ can be interpreted as a stochastic translation of $\overline \bX$ by $\varphi$, cf.\ \cite{TaZa20, BCFP19}.
	Notice however that $\varphi$ does not need to have complementary Young regularity with $X$ (for $H = \frac 14^+$ the regularities sum to $\frac 34^+$),
	so the integrals are only defined canonically thanks to the fact that they are Wiener integrals.
	In particular, $\varphi$ does not need to belong to the Cameron-Martin space, as for the translations considered in Malliavin calculus, see e.g.\ \cite{CF10}.
	The proof below, slightly more generally, classifies all $\mathcal G^{\lfloor 1/H \rfloor}(\mathbb R^d)$-valued $\eta$-H\"older local processes for any $\eta > H/2$.
	It could be interesting to consider more general translations by embedding higher-degree lifts $\bX$ in a rough path space of lower regularity, cf.\ \cite{BFPP22},
	but if this is done at arbitrarily low regularity it would no longer be possible to exclude perturbations of the second level by noise increments.
\end{rem}

\begin{rem}
	Non-existence of certain stochastic integrals has been considered before in \cite{Lyo91}, who treats the example of a deterministic coupling of two scalar Brownian motions.
	The author shows that various specific definitions of path integrals fail to converge for this process due to diverging stochastic area.
	We expect this example, however, not to be an obstruction to our more general notion of local rough path lift, which exists after deterministic renormalisation, cf.\ \cite[Proposition 3.4]{HaiLi22}.
\end{rem}

\begin{proof}[Proof of \autoref{thm:localfbm}]
Let $\bX$ be a square-integrable local $\eta$-RPL of $X$, which for $H \leq 1/4$ will lead to a contradiction. The Chen identity on an arbitrary partition $\pi$ of $[s,t]$ yields
\begin{equation}\label{eq:chen2}
	\boldsymbol X_{s,t}^{\alpha\beta} = \sum_{[u,v] \in \pi} \boldsymbol X_{u,v}^{\alpha\beta} + \sum_{\substack{[u,v], [w,z] \in \pi \\ u < w}} X^\alpha_{u,v} X^\beta_{w,z}
	\;.
\end{equation}
By locality and square-integrability, $\bX$ has Wiener chaos expansion
\begin{equation}\label{eq:Xchaos}
	\boldsymbol X^{\alpha_1\ldots \alpha_n}_{s,t} = \sum_{m = 0}^\infty \sko^m(f_{s,t}^{\alpha_1 \ldots \alpha_n;m}), \qquad f_{s,t}^{\alpha_1 \ldots \alpha_n;m} \in (\cH^{\alpha_1\ldots \alpha_n}_{s,t})^{\odot m}.
\end{equation}
This will only be needed for $n = 1,2,3$, and notice that $f^{\alpha;1}_{s,t} = \mathbbm 1_{[s,t]}^\alpha$.
Suppose now $d\geq 2$ and $\alpha\neq \beta$,
so that $\sum_{\substack{[u,v], [w,z] \in \pi \\ u < w}} X^\alpha_{u,v} X^\beta_{w,z}$ has only a 2nd-chaos component.
Projecting \eqref{eq:chen2} onto $\mathscr W^m$ we obtain the a.e.\ identities
\begin{align*}
	f_{s,t}^{\alpha\beta;m} &= \sum_{[u,v] \in \pi} f^{\alpha\beta;m}_{u,v}\;, \qquad m \neq 2 
	\\
	f_{s,t}^{\alpha\beta;2} &= \sum_{[u,v] \in \pi} f^{\alpha\beta;2}_{u,v} + \sum_{\substack{[u,v], [w,z] \in \pi \\ u < w}} \mathbbm 1^\alpha_{[u,v]} \odot \mathbbm 1^\beta_{[w,z]} \;.
\end{align*}
This implies that, denoting $D_\pi^m = \bigcup_{[u,v] \in \pi} [u,v]^m$, $f^{\alpha\beta;m}_{s,t}$ vanishes outside $D_\pi^m$ for $m > 2$ and any $\pi$.
Moreover, note the identification $(\cH^{\alpha\beta})^{\odot 2} \cong (\cH^{\alpha})^{\odot 2}\oplus (\cH^{\beta})^{\odot 2} \oplus (\cH^\alpha\otimes\cH^\beta)$,
and that $\mathbbm 1^\alpha_{[u,v]} \odot \mathbbm 1^\beta_{[w,z]}$ under this identification corresponds to $\mathbbm 1_{[u,v]}^\alpha\otimes \mathbbm 1_{[w,z]}^\beta\in \cH^\alpha\otimes\cH^\beta$.
Then the component of $f^{\alpha\beta;2}_{s,t}$ in $\cH^\alpha\otimes\cH^\beta$, identified with a function in $L^2([s,t]^2,\R)$, agrees with $\mathbbm 1_{\triangle}$ outside of $D_\pi^2$,
while those in $(\cH^{\alpha})^{\odot 2}\oplus (\cH^{\beta})^{\odot 2}$ vanish outside of $D_\pi^2$, for any $\pi$.
Since $\bigcap_{\pi_k} D_{\pi_k}^m$ is of measure zero for $m \geq 2$ for any sequence of partitions $\pi_k$ with mesh tending to zero, we conclude that
\begin{equation}\label{eq:fabm}
	f^{\alpha\beta;m}_{s,t} =  0 \quad \text{for } m \geq 3\;,
\end{equation}
while $f^{\alpha\beta;2}_{s,t}$ has non-zero component only in $\cH^\alpha\otimes\cH^\beta$ and is equal to $\mathbbm 1_{\triangle[s,t]}$ there.
If $H \leq 1/4$, then $\mathbbm 1_{\triangle}\not\in \cH^\alpha \otimes \cH^\beta$
by \autoref{lem:triangleH}, which is a contradiction, proving \ref{pt:H_less}.
(The restriction $d \geq 2$ was required since otherwise, by symmetry, $f^{\alpha\alpha;2}_{s,t} = \frac12\mathbbm 1_{[s,t]^2} $,
which does belong to $(\cH^{\alpha\alpha}_{s,t})^{\odot 2} = \cH_{s,t}^{\odot 2}$;
cf.\ also \cite[Theorem 3.2]{LN05} for a more general statement regarding symmetrisation in this context.)

We now assume $H > 1/4$ and prove \ref{pt:H_greater} (we no longer assume $d\geq 2$ in the sequel). We proceed with the understanding that third levels are only considered if $H \in (\frac 14, \frac 13]$. Write, similarly to \eqref{eq:Xchaos} (and using locality and square-integrability of $\bX$ and $\overline \bX$)
\begin{equation*}
	\boldsymbol X^{\alpha_1\ldots \alpha_n}_{s,t} - \overline{\boldsymbol X}{}^{\alpha_1\ldots \alpha_n}_{s,t} = \sum_{m = 0}^\infty \sko^m(\varphi_{s,t}^{\alpha_1 \ldots \alpha_n;m}), \qquad \varphi_{s,t}^{\alpha_1 \ldots \alpha_n;m} \in (\cH^{\alpha_1\ldots \alpha_n}_{s,t})^{\odot m},
\end{equation*}
for $n = 1,2,3$. Then (continuing to use round brackets to denote tensor degree, while no brackets are used for Wiener chaos degree) $\varphi^{(1)} = 0$ and the Chen identity for $\bX - \overline \bX$ at level $2$ reads
\begin{equation}\label{eq:chendiff2}
\bX^{(2)}_{s,t} - \overline \bX{}^{(2)}_{s,t} = \bX^{(2)}_{s,u} - \overline \bX{}^{(2)}_{s,u} + \bX^{(2)}_{u,t} - \overline \bX{}^{(2)}_{u,t}.
\end{equation}
We have seen in \eqref{eq:fabm} that the Wiener chaos expansion of $\bX^{(2)}$ and $\overline \bX{}^{(2)}$ stops at the second chaos, and that moreover
the second chaos projections of $\bX^{(2)}$ and $\overline \bX{}^{(2)}$ agree. Thus $\varphi^{(2);m} = 0$ for $m \geq 2$ and \eqref{eq:chendiff2} yields
\[
\varphi_{s,t}^{(2);m} = \varphi_{s,u}^{(2);m} + \varphi_{u,t}^{(2);m}, \qquad m = 0,1.
\]
In particular, $\varphi^{(2);1}_{s,t}$ is the restriction to $[s,t]$ of some element $\varphi^{(2);1}_{0,1} \eqqcolon \overline h{}^1\in\cH$.
Note that for any $H \in (\frac 14, \frac 12]$, $\frac H2 < (\lfloor H^{-1} \rfloor +1)^{-1} < \eta$. The process $\bX^{(2)}_{0,t} - \overline \bX{}^{(2)}_{0,t}$ has a.s.\ $2\eta$-H\"older sample paths where $2\eta>H$
and therefore its components satisfy the assumptions of \autoref{lem:no1}, from which it follows that $\varphi^{(2)} \coloneqq \varphi^{(2);0}$ is $2\eta$-H\"older
and $\varphi^{(2);1}=0$.

Consider now the Chen identity at level $3$ for $\bX - \overline \bX$ with $\pi$ as in \eqref{eq:chen2}:
\begin{equation}\label{eq:chendiff3}
\begin{split}
	\boldsymbol X^{\alpha\beta\gamma}_{s,t} - \overline{\boldsymbol X}{}^{\alpha\beta\gamma}_{s,t} ={} &\sum_{[u,v] \in \pi} ( \boldsymbol X^{\alpha\beta\gamma}_{u,v} - \overline{\boldsymbol X}{}^{\alpha\beta\gamma}_{u,v}) \\
	&+ \sum_{\substack{[u,v], [w,z] \in \pi \\ u < w}} \big( (\boldsymbol X^{\alpha\beta}_{u,v} - \overline{\boldsymbol X}{}^{\alpha\beta}_{u,v}) X^\gamma_{w,z} + X^\alpha_{u,v} (\boldsymbol X^{\beta\gamma}_{w,z} - \overline{\boldsymbol X}{}^{\beta\gamma}_{w,z}) \big) \;.
\end{split}
\end{equation}
Using that the second summand is first chaos-valued (because $\bX^{(2)} - \overline{\bX}{}^{(2)}$ is deterministic), \eqref{eq:chendiff3} projected onto the $m$-th Wiener chaos reads
\begin{equation}\label{eq:chenphi3}
\begin{split}
\varphi^{(3);m}_{s,t} &= \sum_{[u,v] \in \pi} \varphi^{(3);m}_{u,v} \qquad \text{for } m \neq 1 \\
\varphi^{\alpha\beta\gamma;1}_{s,t} &= \sum_{[u,v] \in \pi} \varphi^{\alpha\beta\gamma;1}_{u,v} + \sum_{\substack{[u,v], [w,z] \in \pi \\ u < w}} (\varphi^{\alpha\beta}_{u,v} \mathbbm 1^\gamma_{[w,z]} + \mathbbm 1^\alpha_{[u,v]} \varphi^{\beta\gamma}_{w,z} ) \\
&= \sum_{[u,v] \in \pi} \varphi^{\alpha\beta\gamma;1}_{u,v} + \sum_{[u,v] \in \pi} (\varphi^{\alpha\beta}_{s,u} \mathbbm 1^\gamma_{[u,v]} + \mathbbm 1^\alpha_{[u,v]} \varphi^{\beta\gamma}_{v,t} )
\end{split}
\end{equation}
where as usual we have denoted by $\varphi^{\alpha\beta}$ the components of $\varphi^{(2)} = \varphi^{(2);0}$, and we have used its additivity. As before we conclude $\varphi^{(3);m}$ must vanish for $m \geq 2$. The second summand for $m = 1$ is the discrete approximation \cite[Exercise 2.7.6]{NouPec12} of the Wiener integrand
\[
[s,t] \ni u \mapsto \varphi^{\alpha\beta}_{s,u} \mathbbm 1^\gamma_{[s,t]} + \varphi^{\beta\gamma}_{u,t} \mathbbm 1^\alpha_{[s,t]},
\]
which belongs to $\cH^\alpha_{s,t} \oplus \cH^\gamma_{s,t}$ thanks to the fact that $\varphi^{(2)}_{s,\cdot}$ and $\varphi^{(2)}_{\cdot, t}$ are $2\eta$-H\"older with $2\eta>H$ (again a consequence of additivity)
and $\CHol{2\eta}([s,t]) \subset \cH_{s,t}$ due to \autoref{lem:holder} since 
\[
\eta > \frac{H}{2}, \ H > \frac 14  \implies 2\eta > \frac 12 - H.
\]
Taking limits in $\cH$ for $m = 1$ we obtain
\[
\varphi^{\alpha\beta\gamma;1}_{s,t} = \widetilde \varphi{}^{\alpha\beta\gamma;1}_{s,t} + \varphi^{\alpha\beta}_{s,\cdot} \mathbbm 1^\gamma + \varphi^{\beta\gamma}_{\cdot,t} \mathbbm 1^\alpha
\]
for some $\widetilde \varphi{}^{\alpha\beta\gamma}_{s,t} \in \cH^\alpha_{s,t} \oplus \cH^\beta_{s,t} \oplus \cH^\gamma_{s,t}$ which must be additive, $\widetilde \varphi{}^{\alpha\beta\gamma}_{s,t} = \widetilde \varphi{}^{\alpha\beta\gamma}_{s,u} + \widetilde \varphi{}^{\alpha\beta\gamma}_{u,t}$ (which follows by comparing the discrete approximations), which means that it can be written as $\widetilde \varphi{}^{\alpha\beta\gamma}_{s,t} = \mathbbm 1_{[s,t]} \widetilde \varphi{}^{\alpha\beta\gamma}_{0,1}$. We thus have
\begin{equation}\label{eq:abcholder}
	\bX^{\alpha\beta\gamma}_{s,t} - \overline\bX{}^{\alpha\beta\gamma}_{s,t} - \int_s^t \varphi^{\alpha\beta}_{s,u} \dif X^\gamma_u -  \int_s^t \varphi^{\beta\gamma}_{u,t} \dif X^\alpha_u = \sum_{\lambda \in \{\alpha,\beta,\gamma\}}\int_s^t \widetilde \varphi{}^{\alpha\beta\gamma;1;(\lambda)}_{0,1}(u)\dif X_u^\lambda + \varphi^{\alpha\beta\gamma;0}_{s,t}
	\;.
\end{equation}
We claim that each summand on the left hand side satisfies a $3\eta$-H\"older estimate, uniform in $(s,t) \in \triangle$.
This is clear for $\bX^{\alpha\beta\gamma}_{s,t}$ and $\overline\bX{}^{\alpha\beta\gamma}_{s,t}$ so it remains to argue about the two Wiener integrals. Suppressing indices, we have by \autoref{lem:holder} and hypercontractivity (see e.g.\ \cite[Theorem 2.7.2]{NouPec12})
\[
\|\sko( \varphi_{s,\cdot} \mathbbm 1_{[s,t]}) \|_{L^p} \lesssim_p \|\sko( \varphi_{s,\cdot} \mathbbm 1_{[s,t]}) \|_{L^2} = \| \varphi_{s,\cdot} \|_{\cH_{s,t}} \lesssim_{\eta,H} \| \varphi_{s,\cdot} \|_{\Hol{2\eta}} (t-s)^{2\eta + H}
\;.
\]
Observing that this $2$-parameter process satisfies the Chen identity (as can be immediately verified by linearity of Wiener integration), an identical proof to that of the Kolmogorov continuity criterion for rough paths \cite[Theorem 3.1]{FH20}
(applied with level-$1$ process $(\varphi,X)$, which has inhomogeneous $(2\eta,H^-)$-H\"older regularity),
implies that $\sup_{s<t}(t-s)^{-\rho}|\sko( \varphi_{s,\cdot} \mathbbm 1_{[s,t]})|$ has moments of all orders for any $\rho < 2\eta + H$, and is in particular finite. This proves the claim.

Therefore by \eqref{eq:abcholder} the process $t \mapsto Z_t \coloneqq \int_0^t \widetilde \varphi{}^{\alpha\beta\gamma;1}_\lambda(u)\dif X_u^\lambda + \varphi^{\alpha\beta\gamma;0}_{0,t}$, with a sum over $\lambda$, is $3\eta$-H\"older.
It follows that, for a fixed $\lambda \in \{\alpha,\beta,\gamma\}$,
\begin{align*}
\bigg|\int_s^t \widetilde \varphi{}^{\alpha\beta\gamma;1;(\lambda)}_{0,1}(u)\dif X_u^\lambda + \varphi^{\alpha\beta\gamma;0}_{s,t}\bigg| &= |\E[Z_{s,t} \mid \cF^\lambda]| \leq \E[|Z_{s,t}| \mid \cF^\lambda] \\
&\leq \E\big[\|Z \|_{\Hol{3\eta}} \mid \cF^\lambda\big](t-s)^{3\eta}
\end{align*}
(with the conditional H\"older norm a.s.\ finite) and thus we apply \autoref{lem:no1} once again to the process $t \mapsto \int_0^t \widetilde \varphi{}^{\alpha\beta\gamma;1;(\lambda)}_{0,1}(u)\dif X_u^\lambda + \varphi^{\alpha\beta\gamma;0}_{s,t}$ (no sum over $\lambda$) to conclude that $\widetilde \varphi{}^{\alpha\beta\gamma;1;(\lambda)}_{0,1}=0$ for $\lambda \in \{\alpha,\beta,\gamma\}$ and $\varphi^{\alpha\beta\gamma;0} \in C^{\Hol{3\eta}}_0$.

We must still justify that $(0, \varphi^{(2)}, \varphi^{(3)})$ satisfy the Chen and shuffle identities. For a rough path vanishing at level $1$, the Chen identity is synonymous with additivity of $\varphi^{(2)}$ and $\varphi^{(3)}$, which has already been established. The shuffle identity at degree $2$ is the statement that the symmetric part of $\varphi^{(2)}$ vanishes, which occurs since those of $\bX^{(2)}$ and $\overline \bX{}^{(2)}$ are equal (a consequence of geometricity of $\bX$ and $\overline \bX$). We now check the shuffle relations at degree 3.
First,
\begin{align*}
	\bX^{\alpha\beta}_{s,t} X^\gamma_{s,t} ={}& (\overline\bX{}^{\alpha\beta}_{s,t} + \varphi^{\alpha\beta}_{s,t})X^\gamma_{s,t} \\
	={}& \overline\bX{}^{\alpha\beta\gamma}_{s,t} + \overline\bX{}^{\alpha\gamma\beta}_{s,t} + \overline\bX{}^{\gamma\alpha\beta}_{s,t} + \varphi^{\alpha\beta}_{s,t}X^\gamma_{s,t} \\
	={}& \bX^{\alpha\beta\gamma}_{s,t} + \bX^{\alpha\gamma\beta}_{s,t} + \bX^{\gamma\alpha\beta}_{s,t} + \varphi^{\alpha\beta}_{s,t}X^\gamma_{s,t} \\
	&- \int_s^t \varphi^{\alpha\beta;0}_{s,u} \dif X^\gamma_u - \int_s^t \varphi^{\beta\gamma;0}_{u,t} \dif X^\alpha_u - \int_s^t \varphi^{\alpha\gamma;0}_{s,u} \dif X^\beta_u - \int_s^t \varphi^{\gamma\beta;0}_{u,t} \dif X^\alpha_u \\
	&- \int_s^t \varphi^{\gamma\alpha;0}_{s,u} \dif X^\beta_u - \int_s^t \varphi^{\alpha\beta;0}_{u,t} \dif X^\gamma_u - \varphi^{\alpha\beta\gamma;0}_{s,t} - \varphi^{\alpha\gamma\beta;0}_{s,t} - \varphi^{\gamma\alpha\beta;0}_{s,t} \\
	={}& \bX^{\alpha\beta\gamma}_{s,t} + \bX^{\alpha\gamma\beta}_{s,t} + \bX^{\gamma\alpha\beta}_{s,t} - \varphi^{\alpha\beta\gamma;0}_{s,t} - \varphi^{\alpha\gamma\beta;0}_{s,t} - \varphi^{\gamma\alpha\beta;0}_{s,t}
	\;.
\end{align*}
We thus obtain the desired shuffle relation for $\varphi$ because
\begin{equation}\label{eq:shufflerel}
	\bX^{\alpha\beta}X^\gamma = \bX^{\alpha\beta \,\shuffle\, \gamma}
	\iff \varphi^{\alpha\beta\gamma;0} + \varphi^{\alpha\gamma\beta;0} + \varphi^{\gamma\alpha\beta;0} = 0 
\iff \varphi^{\alpha\beta}\varphi^\gamma = \varphi^{\alpha\beta \shuffle \gamma}
\end{equation}
where the final equivalence is again a consequence of the fact that $\varphi^{(1)} = 0$
and the first statement holds by geometricity of $\bX$ and $\overline \bX$).

Conversely, given a deterministic $\eta$-rough path $(0, \varphi^{(2)}, \varphi^{(3)})$, we must verify regularity, the Chen identity, and the shuffle relations for \(\bX\) defined by \eqref{eq:bX_level2}-\eqref{eq:bX_level3}. The proof of the shuffle relations can be retraced in reverse from \eqref{eq:shufflerel}. Regularity of each summand in $\bX$ is either assumed or has already been argued. Finally, we check the Chen relation directly. Using that the Chen identity holds for $\overline \bX$ and additivity of $\varphi^{(2)}$, the statement at level $2$ is immediate. Similarly, at level $3$, additionally using linearity of the Wiener integral, we have
\begin{align*}
&\bX^{\alpha\beta\gamma}_{s,t} - \bX^{\alpha\beta\gamma}_{s,u} - \bX^{\alpha\beta\gamma}_{u,t}\\
={} &\overline{\bX}{}^{\alpha\beta}_{s,u} X^\gamma_{u,t} + X^\alpha_{s,u} \overline{\bX}{}^{\beta\gamma}_{u,t} + \sko \big( \varphi_{s,\cdot}^{\alpha\beta} \mathbbm 1_{[s,t]}^\gamma - \varphi_{s,\cdot}^{\alpha\beta} \mathbbm 1_{[s,u]}^\gamma - \varphi_{u,\cdot}^{\alpha\beta} \mathbbm 1_{[u,t]}^\gamma \big)  \\
&+ \sko \big( \varphi_{\cdot,t}^{\beta\gamma} \mathbbm 1_{[s,t]}^\alpha - \varphi_{\cdot,u}^{\beta\gamma} \mathbbm 1_{[s,u]}^\alpha - \varphi_{\cdot,t}^{\beta\gamma} \mathbbm 1_{[u,t]}^\alpha \big) + \varphi^{\alpha\beta\gamma}_{s,t} - \varphi^{\alpha\beta\gamma}_{s,u} - \varphi^{\alpha\beta\gamma}_{u,t} \\
={} &\bX^{\alpha\beta}_{s,u} X^\gamma_{u,t} - \varphi^{\alpha\beta}_{s,u}X^\gamma_{u,t} + X^\alpha_{s,u} \bX^{\beta\gamma}_{u,t} - X^\alpha_{s,u}\varphi^{\beta\gamma}_{u,t} + \sko( \varphi^{\alpha\beta}_{s,u} \mathbbm 1^\gamma_{[u,t]}) + \sko(  \mathbbm 1^\alpha_{[s,u]} \varphi^{\beta\gamma}_{u,t}) \\
={} &\bX^{\alpha\beta}_{s,u} X^\gamma_{u,t} + X^\alpha_{s,u} \bX^{\beta\gamma}_{u,t},
\end{align*}
since $\sko( \varphi^{\alpha\beta}_{s,u} \mathbbm 1^\gamma_{[u,t]}) = \varphi^{\alpha\beta}_{s,u}\sko(\mathbbm 1^\gamma_{[u,t]})$ and similar.
\end{proof}

\begin{expl}
Let $X$ be Brownian motion ($H = 1/2$). It may seem like setting, e.g.\ \[
\bX^{\alpha\beta}_{s,t} \coloneqq \overline\bX{}^{\alpha\beta}_{s,t} + \int_s^t X^\alpha_u X^\beta_u \dif u
\] defines a local RPL different to the canonical (Stratonovich) one $\overline \bX$. This is not true however, since $\int_s^t X^\alpha_u X^\beta_u \dif u$ is not measurable w.r.t.\ the sigma algebra $\sigma(X^\alpha_{s,u}, X^\beta_{s,u} : u \in [s,t])$, only w.r.t.\ $\sigma(X^\alpha_u, X^\beta_u : u \in [s,t])$ (the value of this integral depends on the initial value $X_s$). Replacing it with $\int_s^t X^\alpha_{s,u} X^\beta_{s,u} \dif u$ would solve this problem but makes the Chen property fail instead. \autoref{lem:no1} shows that there can be no local increments of $X$ that have the desired regularity.
\end{expl}

Finally, we further constrain the admissible rough path lifts by requiring stationarity or scale-invariance. These are properties of the law of $\bX$ that are respectively the probabilistic analogues of time translation-invariance and parametrisation-invariance introduced in \autoref{sec:determ}. We also consider equality in law under coordinate permutation, to be understood as a spatial isotropy condition. We do not claim that all interesting examples must satisfy these conditions, in particular the third is violated in the case of $3$-dimensional Brownian motion in a magnetic field \cite{FGL15}. For $\overline \bX$ however, they all hold. We will use $\sim$ to denote equality in law.

\begin{cor}[Stationarity, scale-invariance, permutation-invariance]\label{cor:inv}
Let $H \in (1/4,1/2]$, $\eta\in (0,H)$ as in \autoref{thm:localfbm}, and $\bX$ be a local, square-integrable $\eta$-RPL of $X$. Let $\varphi = (0, \varphi^{(2)}, \varphi^{(3)})$ be as in \autoref{thm:localfbm}. In the following, statements regarding third levels should be omitted if $H > \frac 13$.
\begin{enumerate}[label=(\alph*)]
\item \label{pt:stat} $\bX$ is stationary jointly with the underlying noise in the sense that $\bX_{s,t} \sim \bX_{0,t-s}$ if and only if $\varphi_{s,t}^{(n)} = (t-s)A^{(n)}$ with $A \in \mathcal G^{\lfloor 1/H \rfloor}(\mathbb R^d)$ for $n = 2,3$.

\item \label{pt:scale} $\bX$ is scale-invariant in the sense that $\bX^{(n)}_{s,t} \sim (t-s)^{nH} \bX^{(n)}_{0,1}$ for $n = 2,3$ if and only if it is stationary as above and either $H = \frac 12$, $H = \frac 13$ with $A^{(2)} = 0$, or $A = 0$.

\item \label{pt:perm} $\bX$ is coordinate permutation-invariant in the sense that $\sigma_*\bX \sim \bX$ for $\sigma \in \mathfrak S_d$ if and only if $\varphi^{\alpha\beta} = 0$ and
\begin{align}\label{eq:phiequi}
\varphi^{\alpha\beta\gamma} = (\delta^{\alpha\beta} - 2 \delta^{\alpha\gamma} + \delta^{\beta\gamma})\phi
\end{align}
for some $3\eta$-H\"older function $\phi\colon [0,1]\to\R$.
\end{enumerate}
In particular, all three conditions can only hold if either $\varphi = 0$ or for $H = \frac 13$, and in the latter case there is a one-dimensional continuum of such rough paths, namely for $\varphi^{\alpha\beta} = 0$ and $\varphi^{\alpha\beta\gamma}$ as in \eqref{eq:phiequi} with $\phi(t) = \lambda t$, $\lambda \in \mathbb R$.
\end{cor}
\begin{proof}
To prove \ref{pt:stat}, suppose $\bX$ is stationary.
By taking expectations and using that the symmetric part of $\varphi^{(2)}$ vanishes, we obtain $\varphi_{s,t} = \varphi_{0,t-s}$. Using that $\varphi_{s,t} = \varphi_t - \varphi_s$, we have reduced the problem to classifying all additive continuous functions $\varphi \colon [0,1] \to \mathcal G^{\lfloor 1/H \rfloor}(\mathbb R^d) \subset T^3(\mathbb R^d)$. A familiar argument shows that for $\varphi_1 = q \varphi_{1/q}$ and therefore $\varphi_{p/q} = \frac pq \varphi_1$ for all rationals $p/q$, and we obtain that $\varphi$ is of the claimed form by continuity. Conversely, such choices of $\varphi$ are sufficiently regular and
\begin{equation}\label{eq:jointlaw}
\bigg( \overline \bX^{(2)}_{s,t}, \varphi^{(2)}_{s,t}, \overline \bX^{(3)}_{s,t}, \int_s^t\varphi^{(2)}_{s,u} \dif X_u, \varphi^{(3)}_{s,t} \bigg) \sim \bigg( \overline \bX^{(2)}_{0,t-s}, \varphi^{(2)}_{0,t-s}, \overline \bX^{(3)}_{0,t-s}, \int_0^{t-s}\varphi^{(2)}_{0,u} \dif X_u, \varphi^{(3)}_{0,t-s} \bigg),
\end{equation}
by stationarity of increments of fBm and the fact that all components of this random vector are obtained by limiting procedures that are time translation--invariant (e.g.\ $\overline \bX$ is the limit in $L^2$ of Stieltjes integrals of piecewise linear interpolations). This implies the statement.

For \ref{pt:scale}, if $\bX$ is scale-invariant, then by again taking expectations, we obtain 
$\varphi^{(2)}_{s,t} = (t-s)^{2H}\varphi^{(2)}_{0,1}$, which together with additivity forces either $2H = 1$ or $\varphi^{(2)}_{0,1} = 0$. The third level is handled similarly.
Conversely, the fact that these choices indeed preserve scaling is argued analogously to \eqref{eq:jointlaw}.

It remains to prove \ref{pt:perm}.
Let $\mathcal L(\mathbb R^d)$ denote the free Lie algebra over $d$ generators and 
\[
\mathrm{Inv}_{\mathfrak S_d} (\mathbb R^d)^{\otimes n} = \mathrm{span} \Big\{ \frac{1}{d!}\sum_{\sigma \in \mathfrak S_d} e_{\sigma(\alpha_1)} \otimes \cdots \otimes e_{\sigma(\alpha_n)} \ \Big| \ \alpha_1,\ldots, \alpha_n \in [d] \Big\}
\]
the space of tensors that are left invariant by the action of $\mathfrak S_d$ (cf.\ \cite[\S 5]{DR19}).
We first claim that
\begin{equation}\label{eq:L_inv}
\begin{split}
\mathcal L(\mathbb R^d) \cap \mathrm{Inv}_{\mathfrak S_d} (\mathbb R^d)^{\otimes 2} &= \{ 0 \} \\
\mathcal L(\mathbb R^d) \cap \mathrm{Inv}_{\mathfrak S_d} (\mathbb R^d)^{\otimes 3} &= \mathrm{span} \Big\{ \sum_{\alpha\neq\beta} [e_\alpha,[e_\alpha, e_\beta]] \Big\}.
\end{split}
\end{equation}
Indeed, the map $\frac{1}{d!}\sum_{\sigma \in \mathfrak S_d} \sigma_*$ averages algebra morphisms and thus preserves $\mathcal L(\mathbb R^d)$. Therefore, for any $n$
\begin{align*}
\mathcal L(\mathbb R^d) \cap \mathrm{Inv}_{\mathfrak S_d} (\mathbb R^d)^{\otimes n} &= \frac{1}{d!}\sum_{\sigma \in \mathfrak S_d} \sigma_* (\mathcal L^{(n)}(\mathbb R^d)) \\
&= \mathrm{span}\Big\{ \frac{1}{d!}\sum_{\sigma \in \mathfrak S_d} \sigma_* [e_{\alpha_1}, [ e_{\alpha_2}, \ldots [e_{\alpha_{n-1}}, e_{\alpha_n}]\ldots] \ \Big| \ \alpha_1,\ldots, \alpha_n \in [d] \Big\}.
\end{align*}
For level $2$ we conclude immediately. At level $3$, we consider the $5$ orbits:
\[
(1) \ \alpha = \beta = \gamma, \quad (2) \ \alpha \neq \beta = \gamma, \quad (3) \ \alpha = \beta \neq \gamma, \quad (4) \ \alpha = \gamma \neq \beta,\quad (5) \ \alpha \neq \beta \neq \gamma \neq \alpha.
\]
Lie brackets $[e_\alpha, [e_\beta, e_\gamma]]$ corresponding to $(1)$ and $(2)$ vanish by skew-symmetry, the ones corresponding to $(5)$ vanish by Jacobi, and the cases $(3)$ and $(4)$ are linearly dependent by skew-symmetry.
This proves \eqref{eq:L_inv}.

Next, remark that, if $g \in \mathcal G^N(\mathbb R^d)$ for some $N$ and is zero at degrees $k = 1,\ldots,n-1$, then its projection onto $(\mathbb R^d)^{\otimes n}$ must lie in $\mathcal L(\mathbb R^d)$: this follows by writing $g = \exp \log g$ and projecting onto the first non-zero degree.

Returning to the proof of \ref{pt:perm}, suppose $\bX$ is coordinate permutation-invariant. Then $\varphi^{(2)}$ must also be coordinate invariant and thus belongs to $\mathcal L(\mathbb R^d) \cap \mathrm{Inv}_{\mathfrak S_d} (\mathbb R^d)^{\otimes 2}$ by the above remark. Therefore $\varphi^{(2)}=0$ by \eqref{eq:L_inv}.
Likewise $\varphi^{(3)}$ belongs to $\mathcal L(\mathbb R^d) \cap \mathrm{Inv}_{\mathfrak S_d} (\mathbb R^d)^{\otimes 3}$,
from which the claimed form of $\varphi$ follows again by \eqref{eq:L_inv} and the fact that
\[
[e_\alpha,[e_\alpha,e_\beta]] = e_\alpha \otimes e_\alpha \otimes e_\beta - 2 e_\alpha \otimes e_\beta \otimes e_\alpha + e_\beta \otimes e_\alpha \otimes e_\alpha\;.
\]
Conversely, any such $\varphi$ clearly leads to a coordinate permutation-invariant $\bX$.
\end{proof}

\subsection{Dropping the $L^2$ assumption}
\label{sec:drop_L2}

The material in this section depends heavily on the rough path lift being square-integrable, so that it can be expanded in Wiener chaos.
We conjecture that the square-integrability hypothesis in \autoref{thm:localfbm} can be dropped, and moreover that for $H \leq \frac 14$ any local rough path lift defined on a Borel set $B \subset \Omega$ must force $B$ to be null w.r.t.\ the measure of $H$-fBm. For $H> \frac 14$, square-integrability would then follow a posteriori from the resulting classification.


We next show that scale-invariance is sufficient to guarantee square-integrability of a local rough path lift.
The following argument was suggested by GPT-5.6 Sol Pro and subsequently checked and adapted by the authors, who take responsibility for the result.
The proof makes no use of the Chen identity and instead relies on a combination of scale-invariance and H\"older regularity.
We still regard our original conjecture above, which requires no extra properties of the $\eta$-RPL beyond locality, as an interesting open problem.


For $0 \leq s < u < t < v \leq 1$ define
\begin{equation}\label{eq:rho}
	\rho([s,u],[t,v]) \coloneqq \sup \big\{ | \langle h, k \rangle| : h \in \cH^{(d)}_{s,u}, \ k \in \cH^{(d)}_{t,v}, \ \| h \|_{\cH^{(d)}} = 1 = \| k \|_{\cH^{(d)}} \big\}.
\end{equation}
\begin{lem}\label{lem:corr}
	Let $H\in (0,\frac12]$.
	There exists a constant $C_H$ depending only on $H$ such that for $\tau > 0$ with $0 \leq s<s + \tau <t <t+\tau\leq 1$ it holds that
	\[
	\rho([s, s+\tau], [t, t+ \tau]) \leq 1 \wedge C_H\bigg(\frac{\tau}{t-(s + \tau)}\bigg)^{2 - 2H}.
	\]
\end{lem}
\begin{proof}
	The case $H=1/2$ is obvious because the left-hand side is zero.
	Suppose henceforth that $H<1/2$.
	We treat first the case $d = 1$. It follows from approximating $f, g$ with step functions, the inequality $\|f\|_{L^1(I)} \leq |I|^{1/2} \|f\|_{L^2(I)}$ for an interval $I$ (which follows from Cauchy–Schwarz applied to $f \cdot 1$), and \autoref{lem:characterization}
	\begin{align*}
		|\langle f, g \rangle| &\asymp_H \bigg|\int_{[s, s+\tau] \times [t, t+ \tau]} f(u) g(v) (v-u)^{2H - 2} \dif u \dif v \bigg| \\
		&\leq (t-(s + \tau))^{2H - 2}\| f \|_{L^1} \| g \|_{L^1} \\
		&\leq (t-(s + \tau))^{2H - 2} \tau \| f \|_{L^2} \| g \|_{L^2} \\
		&\lesssim_H (t-(s + \tau))^{2H - 2} \tau^{2 - 2H} \| f \|_{\cH_{s,s+\tau}} \| g \|_{\cH_{t, t+\tau}}.
	\end{align*}
	The case of general $d$ follows from an application of Cauchy–Schwarz in $\mathbb R^d$.
\end{proof}

The following version of Gebelein's inequality \cite{Geb41} is very similar to the one presented in \cite[Proposition 5.1]{NPY19}, but differs slightly in its assumptions, so we provide a self-contained proof (cf.\ also \cite{Ver09}). It is a version of the Cauchy–Schwarz inequality between Wiener functionals that takes into account the correlation of the underlying Gaussian fields.

\begin{lem}[Gebelein inequality]\label{lem:gebelein}
	Let $W \colon \mathfrak H \to L^2\Omega$ be an isonormal Gaussian process, and $\mathfrak H_1, \mathfrak H_2$ closed subspaces of the Hilbert space $\mathfrak H$. Assume that $\rho \in [0,1]$ is such that $|\langle h, k \rangle| \leq \rho \| h \| \| k\|$. Then for all centred $F \in L^2(\sigma(W|_{\mathfrak H_1}))$ and $G \in L^2(\sigma(W|_{\mathfrak H_2}))$
	\[
	|\E [FG]| \leq \rho \| F \|_{L^2} \| G \|_{L^2}.
	\]
\end{lem}
\begin{proof}
	Write the Wiener chaos expansions for $F$ and $G$:
	\[
	F = \sum_{m = 1}^\infty \sko^m(f^m), \qquad G = \sum_{m = 1}^\infty \sko^m(g^m),
	\]
	with $f^m \in \mathfrak H_1^{\odot m}$ and $g^m \in \mathfrak H_2^{\odot m}$. Orthogonality of Wiener chaoses gives
	\[
	\E[FG] = \sum_{m = 1}^\infty m! \langle f^m, g^m \rangle_{\mathfrak H^{\odot m}}.
	\]
	Let $T \colon \mathfrak H_1 \to \mathfrak H_2$ be the orthogonal projection of $\mathfrak H$ onto $\mathfrak H_2$, restricted to $\mathfrak H_1$. Then 
	\[
	\langle f^m, g^m \rangle_{\mathfrak H^{\otimes m}} = \langle T^{\otimes m} f^m, g^m \rangle_{\mathfrak H^{\otimes m}}
	\]
	as is easily shown on elementary tensors and extending by continuity of $T^{\otimes m}$. Note that the operator norm of $T$ is bounded by
	\[
	\| T \| = \sup_{\substack{h \in \mathfrak H_1 \\ \| h \| = 1}} \| T h \| = \sup_{\substack{f \in \mathfrak H_1, g \in \mathfrak H_2 \\ \|f\| = 1 = \|g\|}} | \langle f, g \rangle| \leq \rho
	\]
	and by the submultiplicativity of the operator norm for tensor products of bounded operators on Hilbert spaces \cite[Proposition 2.6.12]{KR97}, $\|T^{\otimes m}\| \leq \|T\|^m$. Then by Cauchy–Schwarz
	\[
	|\langle f^m, g^m \rangle| \leq \| T^{\otimes m} f^m \| \| g^m \| \leq \rho^m  \|f^m \| \| g^m \| \leq \rho \|f^m \| \| g^m \|
	\]
	since $0 \leq \rho \leq 1$. Applying Cauchy–Schwarz now in $\ell^2$
	\[
	| \E[FG] | \leq \rho \sum_{m = 1}^\infty m! \|f^m \| \| g^m \| \leq \rho \bigg( \sum_{m = 1}^\infty m! \| f^m \|^2 \bigg)^{1/2}  \bigg( \sum_{m = 1}^\infty m! \| g^m \|^2 \bigg)^{1/2} = \rho \|F\|_{L^2} \|G\|_{L^2}.\qedhere
	\]
\end{proof}

The following elementary lemma will be used to combine the key features of our problem.
\begin{lem}\label{lem:elementary}
Let $S$ be a non-negative, square-integrable random variable, $\delta \in [0,1)$, $a > 0$, and $b, p \geq 0$. Assume that
\[
\E[S] \geq a p, \qquad \mathrm{Var}[S] \leq bp,\qquad \P[S > 0] \leq \delta.
\]
Then 
\[
p \leq \frac{\delta}{1 - \delta} \frac{b}{a^2}.
\]
\end{lem}
\begin{proof}
If $p=0$ the statement is trivial, so we assume $p>0$.
By writing $S = S \mathbbm 1_{S > 0}$ and applying Cauchy–Schwarz we obtain the bound (the $\theta = 0$ case of the Paley–Zygmund inequality)
\[
\delta \geq \P[S > 0] \geq \frac{\E[S]^2}{\E[S^2]} = \frac{\E[S]^2}{\mathrm{Var}[S] + \E[S]^2}.
\]
Rearranging and applying the bounds on mean and variance, we conclude by
\[
\delta bp \geq \delta \mathrm{Var}[S] \geq (1 - \delta) \E[S]^2 \geq (1-\delta)a^2p^2.\qedhere
\]
\end{proof}

We briefly explain the mechanism behind the following integrability argument. Using Gebelein's inequality \autoref{lem:gebelein} applied with the correlation of fractional Brownian motion on disjoint intervals \autoref{lem:corr}, we show an identity for the mean and an upper bound on the variance of a sum $S$ of indicator functions, each on the event that $|\bX_{s,t}|$ exceeds a threshold, the sum taken over a collection of non-adjacent intervals $[s,t]$. The identity on $\E S$ and upper bound on $\mathrm{Var}[S]$ are obtained in terms of a tail probability $p$ of $\bX_{0,1}$, also thanks to scale-invariance. The threshold for $|\bX_{s,t}|$ is carefully scaled so that $S$ is zero with positive probability, enabling the application of \autoref{lem:elementary}, but also so that doing so yields a tail estimate that is strong enough (as parameters involved are varied) to guarantee the needed integrability: this can be done thanks to the H\"older property of $\bX$, again using scale invariance to deduce the claim for all $s,t$.

\begin{thm}\label{thm:scale_integrability}
Let \(H\in(0,\frac12]\) and \(0<\eta<H\) satisfy
\(
\frac1\eta< \lfloor\frac1H \rfloor+1
\).
Let \(\bX\) be a local \(\eta\)-RPL of \(H\)-fBm.
Assume that \(\bX\) is scale-invariant, in the sense that
\[
\bX^{(n)}_{s,t}
\sim
(t-s)^{nH}\bX^{(n)}_{0,1}
\]
for every \(n=2,\ldots,\lfloor\frac1H \rfloor\) and \(0\leq s<t\leq1\). Then
\(
\bX^{(n)}_{s,t}\in L^q
\)
for every \(n=2,\ldots,\lfloor\frac1H \rfloor\) and every
\(
0<q<\frac{1}{n(H-\eta)}.
\)
Consequently, \(\bX\) is square-integrable, and the conclusions of
\autoref{thm:localfbm} remain valid when square-integrability is
replaced by scale-invariance.
\end{thm}

	

\begin{proof}
	Denote $m = \lfloor 1/H \rfloor$.
	We first prove the second claim from the first.
	Note that, for all $n=2,\ldots, m$,
	\[
	2 < \frac{1}{n(H-\eta)} \iff \eta > H - \frac{1}{2n}
	\]
	and the final condition holds because
	\[
	\eta > \frac{1}{m+1} > \frac{1}{2m} = \frac{1}{m} - \frac{1}{2m}\geq H - \frac{1}{2n},
	\]
	where the first inequality holds by assumption.
	Therefore $q=2$ is in the permitted range and it follows from the first claim that
	\(
\bX^{(n)}_{s,t}\in L^2
\).

Fix $n=2,\ldots, m$.
	It remains to prove $\bX^{(n)}_{s,t} \in L^q$ for $0<q < \frac{1}{n(H-\eta)}$.
	Let $N \geq 4$ be even and set for $j = 0,\ldots, N/2 - 1$
	\[
	V_{N,j} \coloneqq N^{nH} \bX^{(n)}_{2j/N, (2j+1)/N}.
	\]
	By locality, $V_{N,j}$ is $\mathcal F_{2j/N, (2j+1)/N}$-measurable, and by scale-invariance $V_{N,j} \sim  \bX^{(n)}_{0,1}$ for all $N, j$. For $x > 0$ define
	\[
	p(x) \coloneqq \P[|\bX^{(n)}_{0,1}| > x],\qquad S_N(x) \coloneqq \sum_{j = 0}^{N/2 - 1} \mathbbm 1_{|V_{N,j}| > x},
	\]
	and note that
	\begin{equation}\label{eq:meanSN}
	\E S_N(x) = Np(x)/2.
	\end{equation}
	Fix $0 \leq i < j \leq N/2 - 1$ and applying the Gebelein inequality \autoref{lem:gebelein} (with Hilbert subspaces $\cH_{2i/N, (2i+1)/N}$ and $\cH_{2j/N, (2j+1)/N}$) to the centred and bounded random variables $\mathbbm 1_{|V_{N,i}| > x} - p(x)$ and $\mathbbm 1_{|V_{N,j}| > x} - p(x)$, with the bound on correlations supplied by \autoref{lem:corr}, we compute
	\begin{align*}
		|\mathrm{Cov}[ \mathbbm 1_{|V_{N,i}| > x},  \mathbbm 1_{|V_{N,j}| > x}]| &\leq \rho_{j-i} p(x)(1 - p(x)) \leq \rho_{j-i} p(x),
	\end{align*}
	where
	\[
	\rho_r \coloneqq \begin{cases} 1 \wedge C_H(2r - 1)^{2H-2} &\text{ if } H<1/2,\\
		0 &\text{ if } H=1/2.
		\end{cases}
	\]
	Then, bounding trivially $\mathrm{Var}[\mathbbm 1_{|V_{N,i}| > x}] = p(x)(1-p(x)) \leq p(x)$,
	\begin{equation}\label{eq:varbound}
	\begin{split}
		\mathrm{Var}[S_N(x)] &\leq Np(x)/2 + 2p(x) \sum_{0 \leq i < j < N/2} \rho_{j-i} \\
		&\leq Np(x)/2 + N p(x) \sum_{r = 1}^{N/2} \rho_r \\[-3ex]
		&\leq B_H N p(x), \qquad \qquad B_H \coloneqq \begin{dcases}\frac 12 +  \sum_{r = 1}^\infty \rho_r &H < \frac 12 \\ \frac 12 &H = \frac 12, \end{dcases}
	\end{split}
	\end{equation}
	where the crude bound on the sum over $i < j$ comes from observing that there are at most $N/2$ summands with $j - i = r$, and
	\[
	\sum_{r = 1}^\infty \rho_r \lesssim_H 1 + \sum_{r = 1}^\infty r^{2H - 2} < \infty 
	\]
	since $H < \frac 12 \implies 2H - 2 < -1$, while if $H = \frac 12$ the off-diagonal covariances vanish altogether.
	
	Let
	\[
	R_\eta = \sup_{0\leq s < t \leq 1} \frac{|\bX^{(n)}_{s,t}|}{|t-s|^{n\eta}}.
	\]
	Since $\bX$ is an $\eta$-H\"older rough path, $R_\eta < \infty$ a.s., and therefore there exist $K_{\eta,\delta}>0$ and $\delta \in [0,1)$ such that
	\begin{equation}\label{eq:holderQuantile}
	\P[  R_\eta > K_{\eta,\delta}] \leq \delta.
	\end{equation}
	By definition, we have 
	\[
		\max_{0 \leq j < N/2} |V_{N,j}| \leq R_\eta N^{n(H - \eta)} \quad \text{a.s.}
	\]
	from which the inclusion
	\[
	\{S_N(K_{\eta,\delta} N^{n(H - \eta)}) > 0\} = \Big\{ \max_{0 \leq j < N/2} |V_{N,j}| > K_{\eta,\delta} N^{n(H - \eta)} \Big\} \subset \big\{ R_\eta > K_{\eta,\delta} \big\}
	\]
	follows. By \eqref{eq:meanSN}, \eqref{eq:varbound} and \eqref{eq:holderQuantile}, the hypotheses of \autoref{lem:elementary} are satisfied with $S = S_N(x_N)$, $a = N/2$, $b = B_HN$, and $p = p(x_N)$ where $x_N = K_{\eta,\delta} N^{n(H - \eta)}$, which implies the tail estimate
	\[
	p(K_{\eta,\delta} N^{n(H - \eta)}) \lesssim_{\delta,H} \frac{1}{N}.
	\]
	This implies that $p(x)\lesssim x^{-\frac{1}{n(H-\eta)}}$ uniformly in $x\geq 1$ and thus $\bX^{(n)}_{0,1} \in L^q$ once $q < \frac{1}{n(H - \eta)}$.
	The claim for general $s,t$ follows from $\bX_{s,t}^{(n)} \sim (t-s)^{nH} \bX^{(n)}_{0,1}$.
	\end{proof}
	\begin{rem}
	For the above proof to work, it is enough to assume that $\bX$ is $\eta$-H\"older with strictly positive probability, since \eqref{eq:holderQuantile} still follows. 
	It may at first appear surprising that controlling the regularity of $\bX$ on an event of positive probability suffices,
	but this is explained by the fact that H\"older regularity is used only to ensure that $\max_{0 \leq j < N/2} |V_{N,j}| \leq K_{\eta,\delta} N^{n(H - \eta)}$ with uniformly positive probability, so that \autoref{lem:elementary} can be applied.
	\end{rem}

\bibliographystyle{alpha} 
\renewcommand\bibname{\sc References}
\bibliography{refs}
\addcontentsline{toc}{section}{References}

\end{document}